\documentclass[twoside, 11pt]{article}
\usepackage{amsmath,amssymb,amsthm,mathrsfs}
\usepackage[margin=2.5cm]{geometry} 
\usepackage{lipsum}
\usepackage{titlesec,hyperref}
\usepackage{fancyhdr}
\usepackage[numbers,sort&compress]{natbib}
\usepackage{color}

\fancypagestyle{plain}
{
\fancyhf{}

}

\titleformat{\subsection}{\it}{\thesubsection.\enspace}{1.5pt}{}
\titleformat{\subsubsection}{\it}{\thesubsubsection.\enspace}{1.5pt}{}

\newtheorem{theo}{Theorem}[section]
\newtheorem{lemm}[theo]{Lemma}
\newtheorem{defi}[theo]{Definition}

\newtheorem{prop}[theo]{Proposition}
\newtheorem{rema}{Remark}[section]
\numberwithin{equation}{section}

\allowdisplaybreaks

\def\th2{\frac{\theta}{2}}

\begin{document}

\title{On the dynamics of Navier-Stokes-Fourier equations\hspace{-4mm}}
\author{ Boling Guo$^1$, Binqiang Xie$^2*$}
\date{}
\maketitle
\begin{center}
\begin{minipage}{120mm}
\emph{\small $^1$Institute of Applied Physics and Computational Mathematics, China Academy of Engineering Physics,
 Beijing, 100088, P. R. China \\
$^2$Graduate School of China Academy of Engineering Physics, Beijing, 100088, P. R. China  }
\end{minipage}
\end{center}

\footnotetext{Email: \it gbl@iapcm.ac.cn(B.L.Guo), \it xbq211@163.com(B.Q.Xie).}
\date{}

\maketitle

\begin{abstract}

In this paper we are concerned with a non-isothermal compressible Navier-Stokes-Fourier model with density dependent viscosity that vanish on the vacuum. We prove the global existence of weak solutions with large data in the three-dimensional torus $\Omega=T^{3}$. The main point is that the pressure is given by $P=R\rho \theta$ without additional cold pressure assumption.

\vspace*{5pt}

\vspace*{5pt}
\end{abstract}


\section{Introduction}
\quad\quad A compressible and heat-conducting fluid governed by the Navier-Stokes-Fourier equations satisfies the following system in $R_{+}\times \Omega$:
\begin{gather}
  \partial_{t}\rho+{\rm div}(\rho u) =0,\label{1.1} \\
  \partial_{t}(\rho u)+{\rm div}(\rho u\otimes u)
  +\nabla P={\rm div}\mathbb{S}, \label{1.2}\\
  \partial_{t} (\rho E) + {\rm div}(\rho E u)+{\rm div} q+ {\rm div}(P u)={\rm div} (\mathbb{S}u),\label{1.3}
\end{gather}
where the functions $\rho, u,\theta $ represent the density,the velocity field, the absolute temperature. $P$ stands for the pressure, $\mathbb{S}$ denotes the viscous stress tensor. $\rho E=\rho e+ \frac{\rho |u|^{2}}{2}$ the total energy, $e$ the internal energy. $q$ the heat flux. Eqs. \eqref{1.1}, \eqref{1.2}, \eqref{1.3} respectively express the conservation of mass, momentum and total energy.

Our analysis is based on the following physically grounded assumptions:
\begin{itemize}
\item The viscosity stress tensor $\mathbb{S}$ is determined by the Newton's rheological law
\begin{equation}\label{1.4}
\mathbb{S}=2\mu(\rho) D(u)+\lambda(\rho) {\rm div}_{x} u \mathbb{I},
\end{equation}
where $3\lambda + 2\mu\geq 0$ and $D(u)= \frac{1}{2}(\nabla u+ \nabla^{T} u)$ denotes the strain rate tensor, we require $\lambda(\rho)= 2(\rho \mu^{\prime}(\rho)- \mu(\rho))$. For simplicity, we only consider a particular case $\mu(\rho)=\rho, \lambda(\rho)=0$.

\item A key element of the system \eqref{1.1}-\eqref{1.3} is pressure $P$, which obeys the following equation of state:
\begin{equation}\label{1.5}
P(\rho,\theta)=R\rho\theta,
\end{equation}
where R is the perfect gas constant, for simplicity, we set $R=1$. This assumption means ideal gas given by Boyle's law.
\end{itemize}

\begin{itemize}
\item In accordance with the second thermodynamics law, the form of the internal energy reads:
\begin{equation}\label{1.6}
e=C_{\nu}\theta ,
\end{equation}
where $C_{\nu}$ is termed the specific heat at constant volume, for simplicity, we set $C_{\nu}=1$.

\item The heat flux $q$ is expressed through the classical Fourier's law:
\begin{equation}\label{1.7}
q=-\kappa\nabla \theta ,
\end{equation}
where the heat conducting coefficient $\kappa$ is assumed to satisfy:
\begin{equation}\label{1.8}
\kappa(\rho,\theta)= \kappa_{0}(\rho,\theta)(1+\theta^{\alpha}),
\end{equation}
where $a\geq 2$, $\kappa_{0}$ is a continuous function of temperature and density satisfying:
$ C_{1} \leq \kappa_{0} (\rho, \theta) \leq \frac{1}{C_{1}}$, for some positive $C_{1}$.
\end{itemize}

Assuming smoothness of the flow the total energy equation \eqref{1.3} can be written using the equation for the thermal energy equation
\begin{equation}\label{1.9}
\partial_{t} (\rho e)+ {\rm div} (\rho e u) +  {\rm div} q=\mathbb{S}: \nabla u -P {\rm div} u .
\end{equation}

Finally, to complete the system \eqref{1.1}-\eqref{1.3}, the initial conditions are given by
\begin{equation}\label{1.10}
\rho(0,\cdot)=\rho_{0}, (\rho u)(0,\cdot)=m_{0}, \theta(0,\cdot)=\theta_{0},
\end{equation}
together with the compatibility condition:
\begin{equation}\label{1.11}
m_{0}=0~~on~the~set~~\{x\in \Omega|\rho_{0}(x)=0\}.
\end{equation}

Now we give the definition of a variational solution to \eqref{1.1}-\eqref{1.10}.

\begin{defi}
We call $(\rho, u, \theta)$ is as  a varational  weak solution to the problem \eqref{1.1}-\eqref{1.10}, if the following is satisfied.

(1)the density $\rho$ is a non-negative function satisfying the internal identity
\begin{equation}\label{1.12}
\int_{0}^{T} \int_{\Omega} \rho \partial_{t} \phi + \rho u\cdot \nabla \phi dx dt+ \int_{\Omega} \rho_{0} \phi(0)dx=0,
\end{equation}
for any test function $\phi\in \mathcal{D}([0,T)\times \overline{\Omega})$.

(2) The momentum equation holds in $D^{\prime}((0,T)\times\Omega)$, that means,
\begin{equation}\label{1.13}
\begin{aligned}
&\int_{\Omega}m_{0} \phi(0) dx+\int_{0}^{T} \int_{\Omega} \rho u \cdot \partial_{t} \phi + \rho (u\otimes u): \nabla \phi+ P {\rm div} \phi dx dt\\
&= \int_{0}^{T} \int_{\Omega} \mathbb{S}: \nabla \phi dx dt,~for~any ~\phi\in \mathcal{D}([0,T)\times \overline{\Omega}),
\end{aligned}
\end{equation}

(3) The temperature $\theta$ is a non-negative function satisfying
\begin{equation}\label{1.14}
\begin{aligned}
&\int_{0}^{T} \int_{\Omega} \rho \theta \partial_{t} \phi + \rho \theta \cdot \phi + \mathcal {K} (\theta) \Delta \phi dx dt \leq \\
&\int_{0}^{T} \int_{\Omega} (R \rho \theta - \mathbb{S}: \nabla u ) \phi dx dt + \int_{\Omega} \rho_{0} \theta_{0} dx=0,
\end{aligned}
\end{equation}
for any $\phi\in C^{\infty}([0,T]\times \Omega)$, $\phi \geq 0$, $\phi(T)=0$, where
\begin{equation*}
\mathcal{K} = \int_{0}^{T} \kappa (z) dz;
\end{equation*}

(4) The total energy inequality holds for a.a. $\tau\in (0,T)$ with
\begin{equation}\label{1.15}
\rho E (\tau) \leq \rho E(0),
\end{equation}
where
\begin{equation*}
\rho E(0) = \int_{\Omega} \frac{1}{2} \frac{|m_{0}|^{2}}{\rho_{0}}+ \rho_{0}\theta_{0};
\end{equation*}
\end{defi}

Now, we are ready to formulate the main result of this paper.

\begin{theo}\label{th1.1}
Let $\Omega$ be the periodic box $T^{3}$.  Assume that
the pressure $P$, the conductivity coefficient $\kappa$ and the viscosity coefficient $\mu$  satisfy the condition \eqref{1.4}-\eqref{1.8}. Assume the initial data $\rho_{0}, u_{0}, \theta_{0}$ satisfy
\begin{equation}\label{1.16}
\rho_{0} \geq 0,~~~\nabla \sqrt{\rho_{0}}\in L^{2}(\Omega),
\end{equation}
\begin{equation}\label{1.17}
\rho_{0} |u_{0}|^{2} \in L^{1}(\Omega), ~~\rho_{0} (1+ |u_{0}|^{2})\ln (1+ |u_{0}|^{2}) < \infty,
\end{equation}
\begin{equation}\label{1.18}
\theta_{0} \in L^{\infty}(\Omega), ~~0 < \underline{\theta} \leq \theta_{0} \leq \overline{\theta} ~~~for ~~a.e.~x\in \Omega.
\end{equation}
Then, for any given $T>0$, there exists a variational weak solution of \eqref{1.1}-\eqref{1.3} on the set $(0,T)\times \Omega$.
\end{theo}

\begin{rema}
Compared with the constant viscosity  and viscosity depending   temperature case stated in \cite{Firesel}, \cite{Firesel2} gives global weak solutions to the nonlinear problem \eqref{1.1}-\eqref{1.3}.
Here the viscosity is depending viscosity.
\end{rema}

\begin{rema}
Compared with the  viscosity depending density case stated in \cite{Desjardin} gives global weak solutions to the nonlinear problem \eqref{1.1}-\eqref{1.3} with additional pressure.
Here the pressure is only ideal gas condition, i.e. $P=R\rho\theta$.
\end{rema}

There is a large amount work on the global existence of weak solutions for the compressible Navier-Stokes
equation, in the constant viscosity coefficients case, one of the main result of the nineties is due to P.L. Lions \cite{Lions}, who proved the global existence of weak solutions for the compressible Navier-Stokes system in the case of barotropic equations of state. Later,
this result has been extended to the somehow optimal case $\gamma>n/2$ in \cite{Novotny} using oscillation defect
measures on density sequences associated with suitable approximation solutions. For the full compressible Navier-Stokes
equation, i.e., including the temperature equation, Feireisl \cite{Firesel} firstly prove the global existence of so-called
variational solutions  for the full compressible Navier-Stokes and heat-conducting system. Later on, he also extended this result to the temperature depending viscosity case\cite{Firesel2}. Such an existence result is obtained for specific pressure laws, given by general pressure equation
\begin{equation*}
P(\rho,\theta)= P_{b}(\rho)+ \theta P_{\theta}(\rho).
\end{equation*}
 Unfortunately, the perfect gas equation of state is not covered by this result. Namely the dominant role of the first, barotropic pressure $P_{b}$ is one of the key argument to obtain such an existence result.

Recently Bresch-Desjardins \cite{Bresch} have made important progress in the case of viscosity
coefficients depending on the density $\rho$, under some structure constraint on the viscosity
coefficients, they discover a new entropy inequality(called BD entropy) which can yield global in time integrability properties
on density gradients. This new structure was first applied in \cite{Desjardins} in the framework of capillary fluid. Later on,
they founded that this BD entropy inequality also can applied in the compressible Navier-Stokes
equation without capillarity. By this new BD entropy inequality, they succeeded in obtaining global existence of weak
solutions in the barotropic fluids with some additional drag terms. However, there are some difficulties
without any additional drag term, as lack of estimates for the velocity. By obtaining
a new apriori estimate on smooth approximation solutions, Mellet-Vasseur \cite{Mellet} study the stability of
barotropic compressible Navier-Stokes equations. Unfortunately, they cannot construct smooth approximation solutions.
Li and Xin \cite{Li} recently have been constructed some suitable approximate system which has smooth solutions satisfying the
energy inequality, the BD entropy inequality, and the Mellet-Vasseur type estimate, therefore they completely solved an
open problem.

As for the density depending viscosities case, the existence of global weal solutions to the Navier-Stokes equations for viscous compressible and heat conducting fluids was firstly proved by D. Bresch and B. Desjardins \cite{Desjardin}. The equation of state is ideal polytropic gas type:
 \begin{equation*}
 P =R\rho \theta+ P_{c}(\rho),
 \end{equation*}
However, they still need additional cold pressure assumption $P_{c}$.  Therefore, Our aim in this work is to remove additional assumption on the equation of state $P_{c}$. In order to prove the global existence of variational weak solutions, we need to construct an adapted approximation scheme and have enough compactness to pass the limit. Suppose we can construct a sequence of approximate solutions $\{\rho_{n}\}_{n=1}^{\infty}, \{u_{n}\}_{n=1}^{\infty}, \{\theta_{n}\}_{n=1}^{\infty}$ we come accross two major stumbling blocks when we passing the limit: (1) the lack of the strong convergence for $\sqrt{\rho_{n}}u_{n}$ in $L^{2}$. (2) concentrations in
$\{\theta_{n}\}_{n=1}^{\infty}$, more specifically, the lack a priori bounds on $\mathcal{K}(\theta_{n})$.

The problem of strong convergence for $\sqrt{\rho_{n}}u_{n}$ in $L^{2}$ can be solved by establishing a Mellet-Vesseur inequality. The inequality  was discovered by Mellet-Vasseur in the baratropic case, providing a $L^{\infty}(0,T; L\log L(\Omega))$ estimate of  $\rho_{n}|u_{n}|^{2}$. However, it is difficult to construct a adapted approximate scheme verifying the B-D entropy inequality and the Mellet-Vesseur inequality. To deal with this issue, we follow the idea in A.F.Vasseur and C.Yu \cite{Vesseur}, \cite{Yu}.  Additional damping terms $r_{0}u, r_{1}\rho |u|^{2}u$ and quantum term $\kappa \rho \nabla (\frac{\Delta \sqrt{\rho}}{\sqrt{\rho}})$ were introduced in their paper.

The problem of temperature concentration lies in the fact that there are only poor a priori bounds on  $\mathcal{K}(\theta_{n})$, specifically,
\begin{equation*}
\mathcal{K}(\theta_{n})~~~are~~~bounded~~~in~~~L^{1}((0,T)\times \Omega),
 \end{equation*}
we adopted an technique from Feireisl where the limit in the sense of renormalized limit of $\mathcal{K}(\theta_{n})$. i.e, the t  hermal energy inequality  is stable with respect to the topology induced by the renormalzied limit.

This paper is organized as follows.
In section $2$, we construct approximate system by adding some term to the origin system and using the Faedo-Galerkin approximation, also we establish the uniform estimates which is independent of $N$ and pass the limit $N\rightarrow \infty$. In section $3$, we deduce the BD entropy energy estimates and pass the limit $\varepsilon \rightarrow 0$.  In section 4, follow the idea in \cite{Yu}, we will get the approximate Mellet-Vasseur inequality for the weka solution.
In section $5$, $6$ and $7$, we recover the original system by vanishing these parameter $m\rightarrow \infty, K\rightarrow \infty, \kappa\rightarrow0, n\rightarrow \infty, r_{0}\rightarrow 0, r_{1}\rightarrow 0$, therefore our main theorem is proved.

\section{Faedo-Galerkin approximation}

In this section we introduce a approximating scheme which involves a system of regularized equations and the Faedo-Galerkin method. More specifically, we follow the idea in [Feireisl]. In begin with, we fix $u_{N}$ in the space $C([0,T]; X_{N})$ and use it to find a unique smooth solution to \eqref{2.1} $\rho=\rho(u_{N})$, then we solve a regularized thermal equation to \eqref{2.3} $\theta=\theta(\rho,u_{N})$, in the following we find
a local solution to the momentum equation by Schauder fixed theorem. Finally, in according with the uniform estimates, we can extend the local solutions for the whole time interval.

We define a finite-dimensional space $X_{N}=span\{e_{1}, e_{2},...,e_{N}\}$, where $N \in \mathbb{N}$, each $e_{i}$ is an orthonormal basic of $L^{2}(\Omega)$ which is also an orthogonal basis of $H^{2}(\Omega)$. We notice that $u\in C^{0}([0,T]; X_{N})$ is given by
\begin{equation*}
u_{N}(t,x)= \sum_{i=1}^{N} \lambda_{i}(t) e_{i}(x), ~~~(t,x)\in [0,T]\times \Omega,
\end{equation*}
for some functions $\lambda_{i}(t)$, and because of all the norms are equivalence on $X_{N}$, hence, u can be bound in $C^{0}([0,T]; C^{k}(\Omega))$ for any $k\geq 0$, thus
\begin{equation*}
\| u_{N} \|_{C^{0}([0,T];C^{k}(\Omega))} \leq \| u_{N} \|_{C^{0}([0,T];L^{2}(\Omega))}.
\end{equation*}
\subsection{Continuity equation}

For any given $u_{N}\in C^{0}([0,T];X_{N})$, by the classical theory of parabolic equations, there exists a smooth solution $\rho $ to the following approximated system
\begin{equation} \label{2.1}
\rho_{t}+ {\rm div}(\rho u_{N})= \varepsilon \Delta \rho,~~~~in~~(0,T)\times \Omega,
\end{equation}
with the initial data
\begin{equation} \label{2.2}
\rho(0,x)=\rho_{0} \geq \nu >0~~~and ~~\rho_{0}(x) \in C^{\infty}(\Omega),
\end{equation}
where $\nu>0$ is a constant. The following lemma can be seen in \cite{Firesel}.

\begin{lemm}
Let $u_{N} \in C([0,T]; X_{N})$ for $N$ fixed and $\rho_{0}$ be as above. Then there exists the unique classical solution to \eqref{2.1}, i.e., $\rho\in V^{\rho}_{[0,T]}$, where
$$ V^{\rho}_{[0,T]}=\left\{
\begin{aligned}
&\rho \in C([0,T]; C^{2+\nu}(\Omega)),\\
&\partial_{t}\rho \in C([0,T]; C^{\nu}(\Omega)),
\end{aligned}
\right.
$$
Moreover, the mapping $u_{N}\mapsto \rho(u_{N})$ maps bounded sets in $C([0,T]; X_{N})$ into bounded sets in $V^{\rho}_{[0,T]}$ and is continuous with values in $C([0,T]; C^{2+\nu^{\prime}}(\Omega))$, $0< \nu^{\prime}<\nu<1$,
\begin{equation*}
\inf_{x\in \Omega} \rho_{0}(x) exp^{-\int_{0}^{T}\|{\rm div} u_{N} \|_{L^{\infty}(\Omega)} ds} \leq
\rho(t,x) \leq \sup_{x\in \Omega} \rho_{0}(x) exp^{-\int_{0}^{T}\|{\rm div} u_{N} \|_{L^{\infty}(\Omega)} ds}
\end{equation*}
Finally, for fixed $N\in \mathbb{N}$, the function $\rho$ is smooth in the space variable.
\end{lemm}

\subsection{Temperature equation}
Next, given $\rho, u_{N}$, the temperature will be looked for as a solution of the approximate thermal energy equation:
\begin{equation}\label{2.3}
\partial_{t} ((\varepsilon+ \rho) \theta) + {\rm div} (\rho \theta u) - \Delta \mathcal{K}(\theta) + \varepsilon \theta^{\alpha+1} = \mathbb{S} : \nabla u- \rho \theta {\rm div} u,
\end{equation}
with
\begin{equation}\label{2.4}
(\varepsilon+ \rho) \theta(0,x)= (\varepsilon+ \rho_{0}) \theta_{0},
\end{equation}
is fulfilled pointwisely on $(0,T)\times \Omega$. Note that we need to regularize the coefficient of \eqref{2.3} with respect to time. A standard approach yields the following result:
\begin{lemm}
Let $u_{N}\in C([0,T]; X_{N})$ be a given vector field and let $\rho(u)$ be the unique solution of \eqref{2.1}.
Then \eqref{2.3} with the initial condition defined as above admits a unique strong solution $\theta=\theta(u_{N})$ which belong to
$$ V^{\theta}_{[0,T]}=\left\{
\begin{aligned}
&\theta  \in L^{\infty}(0,T; W^{1,2}(\Omega)),~~~\theta,\theta^{-1} \in L^{\infty}((0,T)\times \Omega),\\
&\partial_{t}\theta \in L^{2}((0,T)\times \Omega),~~~~~~~~\Delta\theta \in L^{2}((0,T)\times \Omega),
\end{aligned}
\right\}
$$
Moreover, the mapping $u_{N}$ to $\theta(u_{N})$ maps bound sets in $C([0,T]; X_{N})$ into bound sets in $V^{\theta}_{[0,T]}$ and the mapping is continuous with values in $L^{2}(0,T; W^{1,2}(\Omega))$.
\end{lemm}

\subsection{Momentum equation}

The Faedo-Galerkin approximation for the weak formulation of the momentum balance is given by
\begin{equation}\label{2.5}
\begin{aligned}
&\int_{\Omega} \rho u_{N} (T) \psi dx - \int_{\Omega} m_{0} \psi dx+ \varepsilon \int_{0}^{T} \int_{\Omega} \Delta u_{N} \cdot \Delta \psi dx dt\\
&- \int_{0}^{T} \int_{\Omega} (\rho u_{N} \otimes u_{N}) : \nabla \psi dx dt+ \int_{0}^{T} \int_{\Omega} 2\rho \mathbb{D}u_{N}: \nabla \psi dx dt\\
&- \int_{0}^{T} \int_{\Omega} P \nabla \psi dx dt + \varepsilon \int_{0}^{T} \int_{\Omega} \rho^{-10} \nabla \psi dx dt+ \varepsilon \int_{0}^{T} \int_{\Omega} \nabla \rho \cdot \nabla u_{N} \psi dx dt \\
&= -r_{0} \int_{0}^{T} \int_{\Omega} u_{N} \psi dx dt- r_{1} \int_{0}^{T} \int_{\Omega} \rho |u_{N}|^{2} u_{N} \psi dx dt- 2\kappa \int_{0}^{T} \int_{\Omega} \Delta \sqrt{\rho}  \nabla \sqrt{\rho} \psi dx dt\\
& - \kappa \int_{0}^{T} \int_{\Omega} \Delta \sqrt{\rho} \sqrt{\rho} {\rm div} \psi dx dt + \varepsilon \int_{0}^{T} \int_{\Omega} \rho \nabla \Delta^{9} \rho \psi dxdt
\end{aligned}
\end{equation}
for any test function $\psi \in X_{ N}$. The extra term $\varepsilon \Delta^{2} u_{N}$ is not only necessary to extend the local solution obtained by the fixed point theorem to a global one at the Gerlakin level but also to make sure $\partial_{t} (\frac{\nabla \rho}{\rho}) \in L^{2}((0,T)\times \Omega)$ so that it can be as a test function when we compute the B-D entropy at next level, the extra term $\varepsilon \nabla \rho^{-10}$ and $\varepsilon \rho \nabla \Delta^{9} \rho$ are necessary to keep the density bounded, and bounded away from zero for all time. This enables us to take $\frac{\nabla \rho}{\rho}$ as a test function to derive the BD entropy.

Following the same arguments in [6,7,11], we can solve \eqref{2.5} by the fixed point argument. To that purpose, we introduce an operator on the set $\{ \rho \in L^{1}(\Omega), \rho \geq \underline{\rho}>0\}$, where $\underline{\rho}=\xi_{0}$:
\begin{equation*}
\mathcal{M}[\rho(t), \cdot] : X_{N} \rightarrow X_{N}^{*}, ~~~<\mathcal{M}[\rho]u, w>= \int_{\Omega}  \rho u \cdot w dx~~for~u,w\in X_{N},
\end{equation*}
We can show that $\Xi[\rho]$ is invertible,
\begin{equation*}
\|\mathcal{M}^{-1}(\rho)\|_{L(X_{N}^{*},X_{N})} \leq \underline{\rho}^{-1},
\end{equation*}
where $L(X_{N}^{*},X_{N})$ is the set of all bounded linear mappings from $X^{*}_{N}$ to $X_{N}$. It is Lipschitz continuous in the following sense,
\begin{equation*}
\|\mathcal{M}^{-1}(\rho_{1})-\mathcal{M}^{-1}(\rho_{2})\|_{L(X_{N}^{*},X_{N})} \leq C(N,\underline{\rho}) \| \rho_{1}-\rho_{2}\|_{L^{1}(\Omega)},
\end{equation*}
for any $\rho_{1}$ and $\rho_{2}$ from the following set
\begin{equation*}
N_{\nu}=\{\rho\in L^{1}(\Omega)| \inf_{x\in \Omega} \rho \geq \nu >0 \},
\end{equation*}

We also define a mapping
\begin{equation*}
\mathcal{T}: C([0,\tau];X_{N})\rightarrow C([0,\tau];X_{N}), \mathcal{T}(v_{N})=u_{N},
\end{equation*}
them, can rewrite \eqref{2.5} as the following problem:
\begin{equation*}
u_{N}(t)= \mathcal{M}^{-1}[\rho(v_{N})](m^{0}+\int_{0}^{T} P_{X_{N}} \mathcal{N}(v_{N})ds),
\end{equation*}
where
\begin{equation*}
\begin{aligned}
<\mathcal{N}(v_{N}),\phi>&=  \int_{\Omega} (\rho v_{N} \otimes v_{N}) : \nabla \phi dx - \int_{\Omega} 2\rho \mathbb{D}v_{N}: \nabla \phi dx+ \int_{\Omega} P \nabla \phi dx  \\
&+\varepsilon \int_{\Omega} \Delta v_{N} \cdot \Delta \phi dx + \varepsilon \int_{\Omega} \rho^{-10} \nabla \phi dx + \varepsilon \int_{\Omega} \nabla \rho \cdot \nabla v_{N} \phi dx  \\
&+\varepsilon \int_{\Omega} \rho \nabla \Delta^{9} \rho \phi dx -r_{0} \int_{0}^{T} \int_{\Omega} v_{N} \psi dx - r_{1} \int_{\Omega} \rho |v_{N}|^{2} v_{N} \phi dx \\
&- 2\kappa  \int_{\Omega} \Delta \sqrt{\rho}  \nabla \sqrt{\rho} \phi dx - \kappa  \int_{\Omega} \Delta \sqrt{\rho} \sqrt{\rho} {\rm div} \phi dx,
\end{aligned}
\end{equation*}

Next, we consider a ball $\mathcal{B}$ in the space $C([0,T];X_{N})$:
\begin{equation*}
\mathcal{B}_{R,\tau}= \{v\in C([0,T];X_{N}): \|v\|_{C([0,T];X_{N})}\leq R \},
\end{equation*}

It is easier to  show that the operator $\mathcal{T}$ is continuous and maps $\mathcal{B}_{R,\tau}$ into itself, provided $\tau$ is sufficiently small.  Moreover, thanks to lemma 2.1 and 2.2, $\mathcal{T}$ is a continuous mapping and its image consists of Lipschitz functions, thus it is compact in $\mathcal{B}_{R,\tau}$. It allows us to apply the Schauder theorem to infer that there exists at least one fixed point u solving \eqref{2.5} on $[0,\tau]$.

\subsection{Uniform estimates and global-in-time solvability}
In order to extend this solution for the whole time interval [0,T], we need  uniform estimates  of the solution with $N$. Taking $\psi =_{N}$ in \eqref{2.5} and using the approximate continuity equation, we obtain the kinetic energy balance
\begin{equation*}
\begin{aligned}
&\frac{d}{dt} \int_{\Omega} (\frac{1}{2} \rho_{N} |u_{N}|^{2} + \frac{\eta}{10} \rho_{N}^{-10}+ \frac{\kappa}{2}|\nabla \sqrt{\rho_{N}}|^{2} + \frac{\delta}{2}
|\nabla \Delta^{4} \rho_{N}|^{2}) + \varepsilon \int_{\Omega} |\Delta u_{N}|^{2} dx+ \int_{\Omega} \rho_{N} |\mathbb{D} u_{N}|^{2}\\
& + \varepsilon^{2} \int_{\Omega} |\Delta^{5} \rho_{N}|^{2}dx+ \varepsilon^{2} \int_{\Omega} |\nabla \rho^{-5}_{N}|^{2}dx + r_{0} \int_{\Omega} |u_{N}|^{2} dx + r_{1} \int_{\Omega} \rho_{N} |u_{N}|^{4}dx \\
&+ \kappa \varepsilon  \int_{\Omega} \rho_{N} |\nabla^{2} \log \rho_{N}|^{2}dx
= \int_{\Omega} P(\rho_{N},\theta_{N}) {\rm div} u_{N} dx,
\end{aligned}
\end{equation*}
Adding, to this, equality \eqref{2.3} integrated with respect to space and integrating the resulting sum with respect to time we obtain the total energy balance
\begin{equation}\label{2.6}
\begin{aligned}
&\frac{d}{dt} \int_{\Omega} (\frac{1}{2} \rho_{N} |u_{N}|^{2} + \frac{\varepsilon}{10} \rho^{-10}_{N}+ \frac{\kappa}{2}|\nabla \sqrt{\rho_{N}}|^{2} + \frac{\varepsilon}{2}
|\nabla \Delta^{4} \rho_{N}|^{2}+ (\varepsilon+ \rho_{N})\theta_{N})dx + \varepsilon \int_{\Omega} |\Delta u_{N}|^{2} dx\\
& + \varepsilon^{2} \int_{\Omega} |\Delta^{5} \rho|^{2}dx+ \varepsilon^{2} \int_{\Omega} |\nabla \rho^{-5}_{N}|^{2}dx + r_{0} \int_{\Omega} |u_{N}|^{2} dx + r_{1} \int_{\Omega} \rho_{N} |u_{N}|^{4}dx \\
& +   \kappa \varepsilon  \int_{\Omega} \rho_{N} |\nabla^{2} \log \rho_{N}|^{2}dx +  \varepsilon \int_{\Omega} \theta^{\alpha+1}dx
= 0,
\end{aligned}
\end{equation}

Moreover, one can integrate energy equality \eqref{2.6} yields
\begin{equation}\label{2.7}
\int_{0}^{T^{*}} \| \Delta u_{N} \|^{2}_{L^{2}} dt <\infty.
\end{equation}
Due to the equivalence of norms on the finite dimensional of $X_{N}$
, we deduce the uniform bound for $u$ in $C([0,\tau]; X_{N})$. Thus, we can extend local time $\tau$ to global time $T$, i.e. there exists a solution $(\rho, u, \theta)$ to \eqref{2.1}, \eqref{2.3}, \eqref{2.5} for any $T>0$.

\subsection{Estimates independent of N}
Our goal now is to identify a limit $N\rightarrow \infty$ of the approximate solutions $\rho_{N}, u_{N},\theta_{N}$ as a solution of the problem \eqref{2.1}, \eqref{2.3}, \eqref{2.5}. In order to achieve this, additional estimates are needed. In the following compactness analysis, we will always need a lemma proved by J\"{u}ngel \cite{Jungel}.
\begin{prop}
\begin{equation}\label{2.8}
\int_{\Omega} \rho |\nabla^{2} \log \rho|^{2} dx \geq \frac{1}{7} \int_{\Omega} |\nabla^{2} \sqrt{\rho}|^{2} dx,
\end{equation}
and
\begin{equation}\label{2.9}
\int_{\Omega} \rho |\nabla^{2} \log \rho|^{2} dx \geq \frac{1}{8} \int_{\Omega} |\nabla \rho^{\frac{1}{4}}|^{4} dx,
\end{equation}
\end{prop}
By energy equality \eqref{2.6}, we have
\begin{equation}\label{2.10}
\kappa \varepsilon  \int_{\Omega} \rho_{N} |\nabla^{2} \log \rho_{N}|^{2}dx < \infty,
\end{equation}
By Prop 2.3, we have the following uniform estimates:
\begin{equation}\label{2.11}
(\kappa \varepsilon)^{\frac{1}{2}}\|\sqrt{\rho_{N}}\|_{L^{2}(0,T;H^{2}(\Omega))}+
(\kappa \varepsilon)^{\frac{1}{4}}\|\nabla \rho_{N}^{\frac{1}{4}}\|_{L^{4}(0,T;L^{4}(\Omega))} \leq C ,
\end{equation}
where the constant $C>0$ is independent of $N$.

To conclude this part, we have the following lemma on the approximate solutions $(\rho_{N},u_{N},\theta_{N})$.
\begin{prop}
Let $(\rho_{N},u_{N},\theta_{N})$ be the solution of \eqref{2.1}, \eqref{2.3}, \eqref{2.5} on $(0,T)\times \Omega$ constructed above, then we have the following energy inequality
\begin{equation}\label{2.12}
\begin{aligned}
&\sup_{t\in (0,T)} \int_{\Omega} E(\rho_{N},u_{N},\theta_{N}) + \varepsilon \int_{\Omega} |\Delta u_{N}|^{2} dx+ \varepsilon^{2} \int_{\Omega} |\Delta^{5} \rho_{N}|^{2}dx+ \varepsilon^{2} \int_{\Omega} |\nabla \rho_{N}^{-5}|^{2}dx + r_{0} \int_{\Omega} |u_{N}|^{2} dx \\
&+ r_{1} \int_{\Omega} \rho_{N} |u_{N}|^{4}dx  +\varepsilon \int_{\Omega} \theta_{N}^{\alpha+1}dx +   \kappa \varepsilon  \int_{\Omega} \rho_{N} |\nabla^{2} \log \rho_{N}|^{2}dx \leq E_{0}(\rho_{N},u_{N},\theta_{N}),
\end{aligned}
\end{equation}
where
\begin{equation}\label{2.13}
E(\rho_{N},u_{N},\theta_{N})= \int_{\Omega} (\frac{1}{2} \rho_{N} |u_{N}|^{2} + \frac{\varepsilon}{10} \rho_{N}^{-10}+ \frac{\kappa}{2}|\nabla \sqrt{\rho_{N}}|^{2} + \frac{\varepsilon}{2}
|\nabla \Delta^{4} \rho_{N}|^{2}+ (\varepsilon+ \rho_{N})\theta_{N})dx,
\end{equation}
Moreover, we have the following uniform estimates
\begin{equation}\label{2.14}
(\kappa \varepsilon)^{\frac{1}{2}}\|\sqrt{\rho_{N}}\|_{L^{2}(0,T£»H^{2}(\Omega))}+
(\kappa \varepsilon)^{\frac{1}{4}}\|\nabla \rho_{N}^{\frac{1}{4}}\|_{L^{4}(0,T£»L^{4}(\Omega))} \leq C ,
\end{equation}
where the constant $C>0$ is independent of $N$.

In particular, we have the following estimates,
\begin{equation}\label{2.15}
\sqrt{\rho_{N}}u_{N} \in L^{\infty}(0,T; L^{2}(\Omega)), \sqrt{\varepsilon} \Delta u_{N}\in L^{2}((0,T)\times \Omega),
\end{equation}
\begin{equation}\label{2.16}
\varepsilon \Delta^{5} \rho_{N} \in L^{2}((0,T)\times \Omega), \sqrt{\varepsilon} \rho_{N} \in L^{\infty}(0,T;H^{9}(\Omega)),
\sqrt{\kappa} \sqrt{\rho_{N}} \in L^{\infty}(0,T; H^{1}(\Omega)),
\end{equation}
\begin{equation}\label{2.17}
\varepsilon^{\frac{1}{10}}\rho^{-1}_{N} \in L^{\infty}(0,T;L^{10}(\Omega)), \varepsilon \nabla \rho^{-5}_{N} \in L^{2}((0,T)\times \Omega),
\end{equation}
\begin{equation}\label{2.18}
u_{N} \in L^{2}((0,T)\times \Omega), \rho^{\frac{1}{4}}_{N} u_{N} \in L^{4}((0,T)\times \Omega),
\end{equation}
\begin{equation}\label{2.19}
\rho_{N} \theta_{N} \in L^{\infty}(0,T; L^{1} (\Omega)), \theta^{\alpha+1}_{N}  \in  L^{\infty}(0,T; L^{1} (\Omega)),
\end{equation}
\end{prop}

At this stage of approximation, We multiply \eqref{2.3} by $h(\theta_{N})$, where $h$ enjoys the properties such that
\begin{equation}\label{2.20}
\begin{aligned}
&h \in C^{2}[0,\infty), ~~h(0)=1,~~h~~non-increasing ~~on~[0,\infty),~\lim_{z\rightarrow \infty} h(z)=0,\\
& h^{\prime \prime} \geq 2(h^{\prime}(z))^{2}~~~for~all~z\geq 0.
\end{aligned}
\end{equation}
 Accordingly, we obtain
\begin{equation}\label{2.21}
\begin{aligned}
&\partial_{t} ((\varepsilon+ \rho_{N}) Q_{h}(\theta_{N})) + {\rm div} (\rho_{N} Q_{h}(\theta_{N}) u_{N})- \Delta \mathcal{K}_{h}(\theta_{N}) + \varepsilon \theta_{N}^{\alpha+1} h(\theta_{N}) \\
& = h(\theta_{N}) \mathbb{S}:\nabla u_{N} - \kappa(\theta_{N}) h^{\prime}(\theta_{N}) |\nabla \theta_{N}|^{2} - h(\theta_{N}) \rho_{N} \theta_{N} {\rm div} u_{N}\\
&+ \varepsilon \Delta \rho_{N}(Q_{h}(\theta_{N})- \theta_{N} h(\theta_{N})),
\end{aligned}
\end{equation}
where $Q_{h}, \mathcal{K}_{h}$ are determined by
\begin{equation}\label{2.22}
Q_{h}= \int_{0}^{\theta_{N}} h(z) dz,~~~~\mathcal{K}_{h}= \int_{0}^{\theta_{N}} \kappa(z) h(z) dz,
\end{equation}

Integrating \eqref{2.21} over $\Omega$ yields
\begin{equation}\label{2.23}
\begin{aligned}
&\frac{d}{dt} \int_{\Omega} (\varepsilon+ \rho_{N}) Q_{h}(\theta_{N})dx + \varepsilon \int_{\Omega} \theta_{N}^{\alpha+1} h(\theta_{N})dx = \int_{\Omega} h(\theta_{N}) \mathbb{S}:\nabla u_{N} - \kappa(\theta_{N}) h^{\prime}(\theta_{N}) |\nabla \theta_{N}|^{2} dx\\
&+ \int_{\Omega} \varepsilon (\nabla \rho_{N} \cdot \nabla \theta_{N}) \theta_{N} h^{\prime}(\theta_{N}) - \theta_{N} h(\theta_{N}) \rho_{N} \theta_{N} {\rm div} u_{N} dx.
\end{aligned}
\end{equation}

In particular, the choice $h(\theta)= (1+\theta)^{-1}$ leads to relations
\begin{equation*}
-\int_{\Omega} \kappa(\theta_{N}) h^{\prime}(\theta_{N}) |\nabla \theta_{N}|^{2} dx \geq C \int_{\Omega} |\nabla \theta_{N}^{\alpha/2}|^{2} dx ,
\end{equation*}
while
\begin{equation*}
\varepsilon |\int_{\Omega} (\nabla \rho_{N} \cdot \nabla \theta_{N}) \theta_{N} h^{\prime}(\theta_{N})|
\leq \varepsilon \| \nabla \rho_{N} \|_{L^{2}(\Omega)} \| \nabla \theta_{N}^{2}\|_{L^{2}(\Omega)},
\end{equation*}
and
\begin{equation*}
\varepsilon |\int_{\Omega} \theta_{N} h(\theta_{N}) \rho_{N} \theta_{N} {\rm div} u_{N}| \leq C \| \rho_{N} \theta_{N}\|_{L^{2}(\Omega)}\| {\rm div} u_{N}\|_{L^{2}(\Omega)},
\end{equation*}
It follows from hypothesis \eqref{1.8} and the energy estimates \eqref{2.12} that the right-hand side of the last inequality is bounded in $L^{1}(0,T)$ by a constant that depends only on $\delta$.

Consequently, \eqref{2.23} integrated with respect to $t$ together with the energy estimates \eqref{2.12} yield a bound
\begin{equation}\label{2.24}
\|\nabla \log \theta_{N}\|_{L^{2}((0,T)\times \Omega)} \leq C(\varepsilon),~~~\|\nabla \theta_{N}^{\alpha/2} \|_{L^{2}((0,T)\times \Omega)} \leq C(\varepsilon),
\end{equation}
which is independent of $N$.

We note that  both the energy estimates and entropy estimates are independent of $N, \varepsilon$.

\subsection{The first level approximate solutions}
At this stage we are ready to pass to the limit for $N\rightarrow \infty $ in the sequence of approximate solutions $\{\rho_{N}, u_{N}, \theta_{N}\}$ in order to obtain a solution to the system \eqref{2.1}, \eqref{2.3}, \eqref{2.5}. As for uniform estimates of the sequence $\{ \theta_{N} \}$, we need a auxilliary result.

\begin{prop}
Let $\Lambda\geq 1$ a given constant. Let $\rho \geq 0$ be a measurable function satisfying
\begin{equation*}
0<M\leq \int_{\Omega} \rho dx,~~~\int_{\Omega} \rho^{\chi} dx \leq K,
\end{equation*}
for
\begin{equation*}
\chi > \frac{6}{5}.
\end{equation*}
Then there exists a constant $C=C(M,K)$ such that
\begin{equation*}
\|v\|_{L^{2}(\Omega)} \leq C(M,K) (\|\nabla v\|_{L^{2}(\Omega)}+
[\int_{\Omega} \rho |v|^{\frac{1}{\Lambda}}]^{\Lambda}),
\end{equation*}
for any $v\in W^{1,2}(\Omega)$.
\end{prop}

Based on the  previous estimates, we have the following estimates uniform in $N$.
\begin{lemm}
The following estimates hold for any fixed positive constants $\varepsilon, r_{0}, r_{1}$ and $\kappa$:
\begin{equation}\label{2.25}
\| (\sqrt{\rho_{N}})_{t}\|_{L^{2}((0,T)\times \Omega)} + \| \sqrt{\rho_{N}}\|_{L^{2}(0,T; H^{2}(\Omega))} \leq C
\end{equation}
\begin{equation}\label{2.26}
\| (\rho_{N})_{t}\|_{L^{2}((0,T)\times \Omega)} + \|\rho_{N}\|_{L^{2}(0,T; H^{10}(\Omega))} \leq C
\end{equation}
\begin{equation}\label{2.27}
\| (\rho_{N} u_{N})_{t}\|_{L^{2}(0,T; H^{-9}(\Omega))}
 + \|\rho_{N} u_{N}\|_{L^{2}((0,T)\times \Omega)} \leq C
\end{equation}
\begin{equation}\label{2.28}
\nabla (\rho_{N} u_{N}) ~~is~~uniformly~~bounded~~in~~in~~ L^{4}(0,T; L^{\frac{6}{5}}(\Omega))+ L^{2}(0,T; L^{\frac{3}{2}}(\Omega)).
\end{equation}
\begin{equation}\label{2.29}
\| \rho_{N}^{-10}\|_{L^{\frac{5}{3}}((0,T)\times \Omega)}  \leq C
\end{equation}
\begin{equation}\label{2.30}
\| \log\theta_{N}\|_{L^{2}(0,T; W^{1,2}(\Omega))} + \| \theta_{N}^{\frac{\alpha}{2}}\|_{L^{2}(0,T; W^{1,2}(\Omega))} \leq C
\end{equation}
where $C$ is independent of $N$ and depends on $\varepsilon,r_{0},r_{1}, \kappa$.
\end{lemm}

\begin{proof}
The proof of \eqref{2.25}-\eqref{2.29} is same as the Lemma 2.2 in \cite{Yu}.

The estimate \eqref{2.24} together with  \eqref{2.19} make it possible to apply Proposition 2.5 such that \eqref{2.30} hold.
\end{proof}

Applying the Aubin-Lions lemma and Lemma 2.6, we conclude
\begin{equation}\label{2.31}
\rho_{N} \rightarrow \rho~~strongly~in~L^{2}(0,T; H^{9}(\Omega)),~~weakly~in~L^{2}(0,T; H^{10}(\Omega)),
\end{equation}
\begin{equation}\label{2.32}
\sqrt{\rho_{N}} \rightarrow \sqrt{\rho}~~strongly~in~L^{2}(0,T; H^{1}(\Omega)),~~weakly~in~L^{2}(0,T; H^{2}(\Omega)),
\end{equation}
and
\begin{equation}\label{2.33}
\rho_{N} u_{N} \rightarrow \rho u~~strongly~in~L^{2}((0,T)\times \Omega),
\end{equation}
we notice that $u_{N} \in L^{2}((0,T)\times \Omega)$, thus
\begin{equation*}
u_{N} \rightarrow u~~weakly~in~L^{2}((0,T)\times \Omega),
\end{equation*}
Thus we can pass to the limits for the term $\rho_{N} u_{N} \otimes u_{N}$ as follows,
\begin{equation*}
\rho_{N} u_{N}\otimes u_{N} \rightarrow \rho u\otimes u
\end{equation*}
in  the distribution sense.

Here we state the following lemma on the strong convergence of $\rho_{N} |u_{N}|^{2} u_{N}$, which will be used later again. The proof is essentially the same as Lemma 2.3 in \cite{Yu}.
\begin{lemm}
When $N\rightarrow \infty$, we have
\begin{equation*}
\rho_{N} |u_{N}|^{2} u_{N} \rightarrow \rho |u|^{2} u,~~~strongly ~~in~~L^{1}(0,T;L^{1}(\Omega)).
\end{equation*}
\end{lemm}

Meanwhile, we have to mention the following Sobolev inequality
\begin{equation*}
\| \rho^{-1}\|_{L^{\infty}(\Omega)} \leq C (1+ \|\rho\|_{H^{k+2}(\Omega)})^{2} (1+\|\rho^{-1}\|_{L^{3}})^{3}
\end{equation*}
for $k\geq \frac{3}{2}$. Thus the estimates on density from \eqref{2.16}-\eqref{2.17} enable us to use the above inequality to have
\begin{equation}\label{2.34}
\| \rho\|_{L^{\infty}((0,T)\times \Omega)} \geq C(\delta,\eta)>0~~a.e.~~in~~(0,T)\times \Omega.
\end{equation}
\eqref{2.34} and \eqref{2.31} allow us to have $\rho^{-10}_{N}$ converges almost  everywhere to $\rho^{-10}$. Thanks to \eqref{2.29}, we deduce
\begin{equation}\label{2.35}
\rho^{-10}_{N} \rightarrow \rho^{-10}~~strongly~in~L^{1}((0,T)\times \Omega),
\end{equation}

In order to continue, we have to show pointwise convergence of the sequence $\{ \theta_{N} \}$. To this end, we use the fact that the time derivatives $\partial_{t} \theta_{N}$ satisfy the thermal energy inequality.

\begin{lemm}
Let $\{ v_{n}\}_{n=1}^{\infty}$ be a sequence of functions such that
\begin{equation*}
v_{n} ~are~bounded~in~L^{2}(0,T; L^{q}(\Omega)) \cap L^{\infty}(0,T; L^{1}(\Omega)) ,~with~q>\frac{2N}{N+2}.
\end{equation*}
Furthermore, assume that
\begin{equation*}
\partial_{t}v_{n} \geq g_{n}~~~in~\mathcal{D}^{\prime}((0,T)\times \Omega)
\end{equation*}
where
\begin{equation*}
g_{n}~~are~bounded~in~L^{1}(0,T; W^{-m,r}(\Omega))
\end{equation*}
for a certain $m\geq 1$, $r>1$.

Then $\{ v_{n}\}_{n=1}^{\infty}$ contains a subsequence such that
\begin{equation*}
v_{n} \rightarrow v~~in~~L^{2}(0,T; W^{-1,2}(\Omega)).
\end{equation*}
\end{lemm}

Now we want to apply Lemma 2.7 to the sequence $(\varepsilon + \rho_{N} ) \theta_{N} $ appearing in the thermal equation \eqref{2.3}. Note that, in accordance with the estimate \eqref{2.31} for the temperature, we have
\begin{equation}\label{2.36}
\rho_{N} \log \theta_{N}~~bounded~in~L^{2}(0,T; L^{q}(\Omega)) \cap L^{\infty}(0,T; L^{1}(\Omega)) ,~with~q>\frac{2N}{N+2}.
\end{equation}
Thus we can use Lemma 2.7 together with \eqref{2.30} and thermal enery inequality \eqref{2.3} to obtain
\begin{equation*}
(\varepsilon+ \rho_{N}) \theta_{N} \rightarrow (\varepsilon +\rho) \theta (strongly)~~in~L^{2}(0,T;W^{-1,2}(\Omega)).
\end{equation*}
Consequently, in view of $\theta_{N} \in L^{2}(0,T;W^{1,2}(\Omega))$
\begin{equation*}
(\varepsilon+ \rho_{N}) |\theta_{N}|^{2} \rightarrow (\varepsilon +\rho) |\theta|^{2} ~~~~in~[\mathcal{D}^{\prime}((0,T)\times \Omega)]^{N}.
\end{equation*}

As the function $z\mapsto \varepsilon z^{2}+ \rho z^{2}$ is non-decreasing,  this relation allow us to conclude that strong convergence
\begin{equation}\label{2.37}
\theta_{N} \rightarrow \theta ~~~~~~strongly~in~L^{1}((0,T)\times \Omega),
\end{equation}

Now, a simple interpolation argument can be used to deduce  form \eqref{2.37}, \eqref{2.19}, \eqref{2.30}  that
\begin{equation}\label{2.38}
\theta_{N} \rightarrow \theta ~~~~~~strongly~in~L^{p}((0,T)\times \Omega),~~~ for~a~certain~p>\alpha,
\end{equation}
Thus we know that
\begin{equation}\label{2.39}
\theta ~is~strictly~positive~a.e.~on~(0,T)\times \Omega,~~\overline{\log \theta}= \log \theta,~~\overline{\theta^{3}}=\theta^{3},
\end{equation}

Here we state the following lemma on the convergence of $\rho_{N} |u_{N}|^{2} u_{N}$ which is proved in Lemma 2.3 (\cite{Yu}).
\begin{lemm}
When $N\rightarrow \infty$, we have
\begin{equation}\label{2.40}
\rho_{N} |u_{N}|^{2} u_{N} \rightarrow \rho |u|^{2} u~~strongly~in~L^{1}((0,T)\times \Omega),
\end{equation}
\end{lemm}

By the above compactness, we are ready to pass to the limits as $N\rightarrow \infty$ in the approximation system. Thus we have shown that $(\rho,u)$ solves
\begin{equation}\label{2.41}
\partial_{t}\rho+{\rm div}(\rho u) =\varepsilon \Delta \rho,~~~pointwise~~~in~~~(0,T)\times \Omega.
\end{equation}
and for any test function $\psi$ such that the following integral hold:
\begin{equation}\label{2.42}
\begin{aligned}
&\int_{\Omega} \rho u (T) \psi dx - \int_{\Omega} m_{0} \psi dx+ \mu \int_{0}^{T} \int_{\Omega} \Delta u \cdot \Delta \psi dx dt\\
&- \int_{0}^{T} \int_{\Omega} (\rho u \otimes u) : \nabla \psi dx dt+ \int_{0}^{T} \int_{\Omega} 2\rho \mathbb{D}u: \nabla \psi dx dt\\
&- \int_{0}^{T} \int_{\Omega} R\rho \theta \nabla \psi dx dt + \eta \int_{0}^{T} \int_{\Omega} \rho^{-10} \nabla \psi dx dt+ \varepsilon \int_{0}^{T} \int_{\Omega} \nabla \rho \cdot \nabla u \psi dx dt \\
&= -r_{0} \int_{0}^{T} \int_{\Omega} u \psi dx dt- r_{1} \int_{0}^{T} \int_{\Omega} \rho |u|^{2} u \psi dx dt- 2\kappa \int_{0}^{T} \int_{\Omega} \Delta \sqrt{\rho}  \nabla \sqrt{\rho} \psi dx dt\\
& - \kappa \int_{0}^{T} \int_{\Omega} \Delta \sqrt{\rho} \sqrt{\rho} {\rm div} \psi dx dt + \delta \int_{0}^{T} \int_{\Omega} \rho \nabla \Delta^{9} \rho \psi dxdt
\end{aligned}
\end{equation}
Thanks to the weak lower semicontinuity of convex functions, we are able to pass to the limits in the energy inequality \eqref{2.12}; by the strong convergence of the density and temperature, we have the following energy inequality in the sense of distributions on $(0,T)$:
\begin{equation}\label{2.43}
\begin{aligned}
&\sup_{t\in (0,T)} \int_{\Omega} E(\rho,u,\theta) + \varepsilon \int_{\Omega} |\Delta u|^{2} dx+ \varepsilon^{2}\int_{\Omega} |\Delta^{5} \rho|^{2}dx+ \varepsilon^{2} \int_{\Omega} |\nabla \rho^{-5}|^{2}dx + r_{0} \int_{\Omega} |u|^{2} dx \\
&+ r_{1} \int_{\Omega} \rho |u|^{4}dx  +   \kappa \varepsilon  \int_{\Omega} \rho |\nabla^{2} \log \rho|^{2}dx
+ \varepsilon \int_{\Omega} \theta^{\alpha+1}dx \leq E_{0}(\rho,u,\theta),
\end{aligned}
\end{equation}

where
\begin{equation}\label{2.44}
E(\rho,u,\theta)= \int_{\Omega} (\frac{1}{2} \rho |u|^{2} + \frac{\eta}{10} \rho^{-10}+ \frac{\kappa}{2}|\nabla \sqrt{\rho}|^{2} + \frac{\delta}{2}
|\nabla \Delta^{4} \rho|^{2}+ (\varepsilon+\rho)\theta )dx,
\end{equation}

Finally, we will pass to the limit for $N\rightarrow \infty$ in \eqref{2.21} to obtain \eqref{2.47}. Note that it is enough to show that one can pass to the limit in all non-linear terms contained in (\eqref{2.47}. To this end, we have used weak lower-continuity of the dissipative estimate:
To begin with,  we can use \eqref{2.42} together with estimates \eqref{2.33}, \eqref{2.31}, \eqref{2.19}, \eqref{2.36} to deduce
\begin{equation}\label{2.45}
(\varepsilon + \rho_{N})  Q_{h} (\theta_{N}) \rightarrow
(\varepsilon + \rho) Q_{h}(\theta)~~~in~~ L^{1}((0,T)\times \Omega)
\end{equation}
and
\begin{equation}\label{2.46}
\rho_{N}  Q_{h} (\theta_{N}) u_{N} \rightarrow \rho  Q_{h} (\theta) u~~weakly~in~~ L^{r}((0,T)\times \Omega)
\end{equation}
and
\begin{equation}\label{2.47}
\rho_{N} \theta_{N} h (\theta_{N})  {\rm div} u_{N} \rightarrow \rho \theta h (\theta)  {\rm div} u~~weakly~in~~ L^{r}((0,T)\times \Omega)
\end{equation}
for a certain $r>1$.

Moreover, because of convexity of the function
$$ [\mathbb{M},\theta]\mapsto\left\{
\begin{aligned}
&h(\theta)(\frac{\mu}{2}\mathbb{M}:\mathbb{M}+ \lambda (tr[\mathbb{M}])^{2}),~~~if~\theta\geq 0, ~~\mathbb{M} \in R^{N^{2}},\\
&\infty,~~~~~~~~~~~~~~~~~~~if~\theta< 0,
\end{aligned}
\right\}
$$
we get
\begin{equation}\label{2.48}
\int_{0}^{T} \int_{\Omega} h(\theta) \mathbb{S}: \nabla u \psi dx dt \leq \lim \inf_{N\rightarrow \infty} \int_{0}^{T} \int_{\Omega}h(\theta_{N}) \mathbb{S}: \nabla u_{N}\psi dx dt,
\end{equation}
for any non-negative test function $\psi$. Similarly,
\begin{equation}\label{2.49}
-\int_{0}^{T} \int_{\Omega}\psi \kappa(\theta) h^{\prime}(\theta) |\nabla \theta|^{2} dx dt \leq \lim \inf_{N\rightarrow \infty} \int_{0}^{T} \int_{\Omega}\psi \kappa(\theta_{N}) h^{\prime}(\theta_{N})|\nabla \theta_{N}|^{2} dx dt,
\end{equation}

Now, because of strong convergence of $\nabla \rho_{N}$ established in \eqref{2.33}, we
\begin{equation}\label{2.50}
\begin{aligned}
&\int_{0}^{T} \int_{\Omega} \varepsilon \nabla(\psi(\log\theta_{N}-1)) \cdot \nabla \rho_{N} + \psi \rho_{N} {\rm div} u_{N} dx dt\\
& \rightarrow \int_{0}^{T} \int_{\Omega} \varepsilon \nabla(\psi(\log\theta-1)) \cdot \nabla \rho + \psi \rho {\rm div} u dx dt
\end{aligned}
\end{equation}
Finally, by virtue of \eqref{2.40}, \eqref{2.41}, \eqref{2.33}, \eqref{2.31}, \eqref{2.19}, \eqref{2.36}
\begin{equation}\label{2.51}
\mathcal{K}_{h}(\theta_{N}) \rightarrow \mathcal{K}_{h}(\theta)~~~in~~L^{1}((0,T)\times \Omega),
\end{equation}
\begin{equation}\label{2.52}
h(\theta_{N}) \theta^{\alpha+1}_{N}  \rightarrow h(\theta) \theta^{\alpha+1}~~~in~~L^{1}((0,T)\times \Omega),
\end{equation}

Making use of these estimates \eqref{2.45}-\eqref{2.52} we are able to let $N \rightarrow \infty$ in \eqref{2.21} in order to obtain a renormalized thermal energy inequality:
\begin{equation}\label{2.53}
\begin{aligned}
&\int_{0}^{T} \int_{\Omega}   ((\varepsilon+ \rho) Q_{h}(\theta))  \partial_{t} \psi+ (\rho Q_{h}(\theta) u) \cdot \nabla \psi+ \Delta \mathcal{K}_{h}(\theta) \Delta \psi - \varepsilon \theta^{\alpha+1} h(\theta) \psi dx dt \\
&\leq \int_{0}^{T} \int_{\Omega}  (\kappa(\theta_{N}) h^{\prime}(\theta_{N}) |\nabla \theta_{N}|^{2} -h(\theta_{N}) \mathbb{S}:\nabla u_{N}) dx dt +\int_{0}^{T} \int_{\Omega} h(\theta_{N}) \rho_{N} \theta_{N} {\rm div} u_{N} dx dt\\
&+ \varepsilon \int_{0}^{T} \int_{\Omega} \Delta \rho_{N}(Q_{h}(\theta_{N})- \theta_{N} h(\theta_{N})) dx dt - \int_{\Omega} (\varepsilon + \rho_{0,N}) Q_{h}(\theta_{0,N}) dx,
\end{aligned}
\end{equation}
to be satisfied for any test function
\begin{equation*}
\psi \in C^{\infty}([0,T]\times \Omega),~~~\psi \geq 0,~~~\psi(0)=1,~~~\psi(T)=0,
\end{equation*}

\section{BD entropy and vanishing limits $\varepsilon \rightarrow 0$}
The goal of this section is to pass into the limits for $\varepsilon \rightarrow 0$ in the family of approximate solutions $\{\rho_{\varepsilon}, u_{\varepsilon}, \theta_{\varepsilon}\}$ constructed in Section 2. In order to achieve this task, we will  deduce the BD entropy for the approximation system in Section 2.  By \eqref{2.3888} and \eqref{2.46}, we have
\begin{equation}\label{3.1}
\rho_{\varepsilon} \geq C(\varepsilon)>0,~~and ~~\rho_{\varepsilon} \in L^{2}(0,T; H^{10}(\Omega))\cap L^{\infty}(0,T; H^{9}(\Omega)).
\end{equation}

\subsection{BD entropy}
Thanks to \eqref{3.1}, we can use $\psi= \nabla(\log \rho_{\varepsilon})$ to test the momentum equation to derive the BD entropy. Thus we have the following lemma.

\begin{lemm}
\begin{equation}\label{3.2}
\begin{aligned}
&\frac{d}{dt} \int_{\Omega} (\frac{1}{2} \rho |u_{\varepsilon}+ \frac{\nabla \rho_{\varepsilon}}{\rho_{\varepsilon}}|^{2} + \frac{\varepsilon}{10} \rho_{\varepsilon}^{-10}+ \frac{\kappa}{2}|\nabla \sqrt{\rho_{\varepsilon}}|^{2} + \frac{\delta}{2}
|\nabla \Delta^{4} \rho_{\varepsilon}|^{2})dx
+\eta \int_{\Omega} |\nabla \rho_{\varepsilon}^{-5}|^{2}dx\\
& + \kappa   \int_{\Omega} \rho_{\varepsilon} |\nabla^{2} \log \rho_{\varepsilon}|^{2}dx+ 2 \varepsilon \int_{\Omega} |\Delta^{5} \rho_{\varepsilon}|^{2}dx+
\frac{1}{2} \int_{\Omega} \rho_{\varepsilon} |\nabla u- \nabla^{T} u_{\varepsilon}|^{2}dx+ \varepsilon
\int_{\Omega} \frac{|\Delta \rho|^{2}}{\rho_{\varepsilon}} dx \\
&+ \int_{\Omega} \frac{|\nabla \rho_{\varepsilon}|^{2}}{\rho_{\varepsilon}} \theta_{\varepsilon}= \varepsilon \int_{\Omega}  \nabla \rho_{\varepsilon} \cdot \nabla u_{\varepsilon} \cdot \nabla \log \rho dx+ \varepsilon \int_{\Omega} \Delta \rho_{\varepsilon} \frac{|\nabla \log \rho_{\varepsilon}|^{2}}{\rho_{\varepsilon}} dx- \varepsilon \int_{\Omega} {\rm div}(\rho_{\varepsilon} u_{\varepsilon}) \frac{1}{\rho_{\varepsilon}} \Delta \rho_{\varepsilon} dx\\
& -\varepsilon\int_{\Omega} \Delta u_{\varepsilon} \cdot \nabla \Delta \log \rho_{\varepsilon} dx
- r_{0} \int_{\Omega} \frac{u_{\varepsilon}\cdot \nabla\rho_{\varepsilon}}{\rho_{\varepsilon}} dx
- r_{1} \int_{\Omega}|u_{\varepsilon}|^{2} u_{\varepsilon} \nabla \rho_{\varepsilon} dx -\int_{\Omega} R \rho_{\varepsilon} \theta_{\varepsilon} {\rm div} u_{\varepsilon}  dx \\
&- \int_{\Omega} \nabla  \theta_{\varepsilon} \nabla \rho_{\varepsilon} dx =R_{1}+ R_{2}+ R_{3}+R_{4}+ R_{5}+ R_{6} +R_{7}+ R_{8},
\end{aligned}
\end{equation}
\end{lemm}

We follow the same arguments in [16] to control terms $R_{i}$ for $i=1,2,3,4,5,6$, and they approach to zero as $\varepsilon\rightarrow 0$ or $\mu\rightarrow 0$ or $r_{0}\rightarrow 0$ or $r_{1}\rightarrow 0$. We estimate $R_{7}$ as follows:
\begin{equation} \label{3.3}
\begin{aligned}
&|R_{7}| \leq \varepsilon \int_{\Omega} \rho_{\varepsilon} |{\rm div} u_{\varepsilon}|^{2}dx + C(\varepsilon) \int_{\Omega} \int_{\Omega} \rho_{\varepsilon} \theta_{\varepsilon}^{2} dx\\
& \leq \varepsilon \int_{\Omega} \rho_{\varepsilon} |{\rm div} u_{\varepsilon}|^{2}dx + C(\varepsilon) \|\theta_{\varepsilon}\|_{L^{3}}^{2} \| \nabla \sqrt{\rho_{\varepsilon}}\|_{L^{2}}^{2},
\end{aligned}
\end{equation}
and for $R_{8}$, we have
\begin{equation} \label{3.4}
\begin{aligned}
&|R_{8}| \leq C\int_{\Omega} \frac{\rho_{\varepsilon} \theta_{\varepsilon}^{2}}{\kappa} |\nabla \sqrt {\rho_{\varepsilon}}|^{2}dx + C \int_{\Omega} \frac{\kappa |\nabla \theta_{\varepsilon}|^{2}}{\theta_{\varepsilon}^{2}}\\
& \leq  C \int_{\Omega}  |\nabla \sqrt {\rho_{\varepsilon}}|^{2}dx + C,
\end{aligned}
\end{equation}

Thus,  by taking $\varepsilon$ small enough, \eqref{3.2}-\eqref{3.4} and Sobolev inequality, $\theta_{\varepsilon}\in L^{2}([0,T]; L^{6}(\Omega))$, it is possible to get some a priori estimates via Gronwall's inequality.
Therefore, we have the following inequality
\begin{lemm}
\begin{equation}\label{3.5}
\begin{aligned}
&\int_{\Omega} (\frac{1}{2} \rho_{\varepsilon} |u_{\varepsilon}+ \frac{\nabla \rho_{\varepsilon}}{\rho_{\varepsilon}}|^{2} + \frac{\varepsilon}{10} \rho_{\varepsilon}^{-10}+ \frac{\kappa}{2}|\nabla \sqrt{\rho_{\varepsilon}}|^{2} + \frac{\varepsilon}{2}
|\nabla \Delta^{4} \rho_{\varepsilon}|^{2}-r_{0} \log \rho_{\varepsilon})dx
+\varepsilon \int_{0}^{T}\int_{\Omega} |\nabla \rho_{\varepsilon}^{-5}|^{2}dx\\
& + \kappa  \int_{0}^{T} \int_{\Omega} \rho_{\varepsilon} |\nabla^{2} \log \rho_{\varepsilon}|^{2}dx+ 2 \varepsilon \int_{0}^{T}\int_{\Omega} |\Delta^{5} \rho_{\varepsilon}|^{2}dx+
\frac{1}{2} \int_{0}^{T}\int_{\Omega} \rho_{\varepsilon} |\nabla u_{\varepsilon}- \nabla^{T} u_{\varepsilon}|^{2}dx+ \varepsilon
\int_{0}^{T}\int_{\Omega} \frac{|\Delta \rho_{\varepsilon}|^{2}}{\rho_{\varepsilon}} dx \\
&+ \int_{\Omega} \frac{|\nabla \rho_{\varepsilon}|^{2}}{\rho_{\varepsilon}} \theta\leq
\int_{\Omega} (\frac{1}{2} \rho_{0,\varepsilon} |u_{0,\varepsilon}+ \frac{\nabla \rho_{0}}{\rho_{0}}|^{2} + \frac{\varepsilon}{10} \rho^{-10}_{0,\varepsilon}+ \frac{\kappa}{2}|\nabla \sqrt{\rho_{0,\varepsilon}}|^{2} + \frac{\varepsilon}{2} |\nabla \Delta^{4} \rho_{0,\varepsilon}|^{2})dx + 2E_{0},
\end{aligned}
\end{equation}
\end{lemm}

Then, we infer the following estimate from the BD entropy:
\begin{equation*}
\kappa  \int_{0}^{T} \int_{\Omega} \rho |\nabla^{2} \log \rho_{\varepsilon}|^{2}dx \leq C
\end{equation*}
where $C$ is independent of $\varepsilon$.

Applying Lemma 2.1, we have the following uniform estimate
\begin{equation}\label{3.6}
(\kappa )^{\frac{1}{2}}\|\sqrt{\rho_{\varepsilon}}\|_{L^{2}(0,T; H^{2}(\Omega))}+
(\kappa )^{\frac{1}{4}}\|\nabla \rho_{\varepsilon}^{\frac{1}{4}}\|_{L^{4}(0,T; L^{4}(\Omega))} \leq C ,
\end{equation}
where $C$ is independent of $\varepsilon$.

\subsection{Uniform estimates with $\varepsilon$.}
From the energy estimate \eqref{2.46}, we have the following uniform estimates on $(\rho_{\varepsilon}, u_{\varepsilon}, \theta_{\varepsilon})$:
\begin{equation}\label{3.7}
\sqrt{\rho_{\varepsilon}}u_{\varepsilon} \in L^{\infty}(0,T; L^{2}(\Omega)),~~\sqrt{\rho_{\varepsilon}} \mathbb{D}u_{\varepsilon,} \in L^{2}((0,T)\times \Omega), ~~\sqrt{\varepsilon} \Delta u_{\varepsilon}\in L^{2}((0,T)\times \Omega),
\end{equation}
\begin{equation}\label{3.8}
\varepsilon \Delta^{5} \rho_{\varepsilon} \in L^{2}((0,T)\times \Omega), \sqrt{\varepsilon} \rho_{\varepsilon} \in L^{\infty}(0,T;H^{9}(\Omega)),
\sqrt{\kappa} \sqrt{\rho_{\varepsilon}} \in L^{\infty}(0,T; H^{1}(\Omega)),
\end{equation}
\begin{equation}\label{3.9}
\varepsilon^{\frac{1}{10}}\rho^{-1}_{\varepsilon} \in L^{\infty}(0,T;L^{10}(\Omega)), \varepsilon  \nabla \rho^{-5}_{\varepsilon} \in L^{2}((0,T)\times \Omega),
\end{equation}
\begin{equation}\label{3.10}
u_{\varepsilon} \in L^{2}((0,T)\times \Omega), \rho^{\frac{1}{4}}_{\varepsilon} u_{\varepsilon} \in L^{4}((0,T)\times \Omega),
\end{equation}
\begin{equation}\label{3.11}
\rho\theta_{\varepsilon} \in L^{\infty}(0,T; L^{1} (\Omega)), \varepsilon \theta_{\varepsilon}^{\alpha+1}  \in L^{\infty}(0,T; L^{1} (\Omega)),
\end{equation}
Moreover, by the BD entropy, we have
\begin{equation}\label{3.12}
\nabla \sqrt{\rho_{\varepsilon}} \in L^{\infty}(0,T; L^{2}(\Omega)),~~\sqrt{\varepsilon} \Delta^{5} \rho_{\varepsilon} \in L^{2}((0,T)\times \Omega),
\end{equation}
and
\begin{equation}\label{3.13}
\nabla \rho^{\frac{\gamma}{2}}_{\varepsilon} \in L^{2}((0,T)\times \Omega),~~\sqrt{\varepsilon} \nabla \rho^{-5}_{\varepsilon,\mu}\in L^{2}((0,T)\times \Omega).
\end{equation}
Also, we have
\begin{equation}\label{3.14}
(\kappa )^{\frac{1}{2}}\|\sqrt{\rho_{\varepsilon}}\|_{L^{2}(0,T; H^{2}(\Omega))}+
(\kappa )^{\frac{1}{4}}\|\nabla \rho^{\frac{1}{4}}_{\varepsilon}\|_{L^{4}(0,T; L^{4}(\Omega))} \leq C ,
\end{equation}
where $C$ is independent of $\varepsilon$.

In according with Lemma 3.2, one deduces
\begin{equation}\label{3.15}
\int_{0}^{T} \int_{\Omega} \rho_{\varepsilon} |\nabla u_{\varepsilon}- \nabla^{T} u_{\varepsilon}|^{2} \leq C
\end{equation}
which together with \eqref{3.7}, yields
\begin{equation}\label{3.16}
\int_{0}^{T} \int_{\Omega} \rho_{\varepsilon} |\nabla u_{\varepsilon}|^{2} \leq C,
\end{equation}
where $C$ is independent of $\varepsilon$. Based on the above estimates, we have the following lemma.
\begin{lemm}
The following further uniform estimates independent of $\varepsilon$ hold:
\begin{equation}\label{3.17}
\| (\sqrt{\rho_{\varepsilon}})_{t}\|_{L^{2}((0,T)\times \Omega)} + \| \sqrt{\rho_{\varepsilon}}\|_{L^{2}(0,T; H^{2}(\Omega))} \leq C
\end{equation}
\begin{equation}\label{3.19}
\| (\rho_{\varepsilon} u_{\varepsilon})_{t}\|_{L^{2}(0,T; H^{-9}(\Omega))}
 + \|\rho_{\varepsilon} u_{\varepsilon}\|_{L^{2}((0,T)\times \Omega)} \leq C
\end{equation}
\begin{equation}\label{3.20}
\nabla (\rho_{\varepsilon} u_{\varepsilon}) ~~is~~uniformly~~bounded~~in~~in~~ L^{4}(0,T; L^{\frac{6}{5}}(\Omega))+ L^{2}(0,T; L^{\frac{3}{2}}(\Omega)).
\end{equation}
\begin{equation}\label{3.21}
\| \rho_{\varepsilon}^{-10}\|_{L^{\frac{5}{3}}((0,T)\times \Omega)}  \leq C
\end{equation}
where $C$ is independent of $\varepsilon$ and depends on $r_{0},r_{1}, \kappa$.
\end{lemm}
\begin{proof}
By \eqref{3.7}-\eqref{3.16}, following the same path as in the proof of Lemma 2.2, we can prove the above estimates.
\end{proof}

\subsection{Temperature estimate}
Now taking
\begin{equation*}
\psi(t,x)= \varphi(t),~~~0\leq \varphi \leq 1,~~~\varphi \in \mathcal{D}(0,T),~~h(\theta)= \frac{\omega}{\omega+\theta},~~\omega>0,
\end{equation*}
in \eqref{2.53}, we deduce
\begin{equation}\label{3.22}
\begin{aligned}
&\int_{0}^{T} \int_{\Omega} (\frac{1}{\omega+\theta_{\varepsilon}} \mathbb{S}_{\varepsilon}: \nabla u_{\varepsilon} + \frac{\kappa(\theta_{\varepsilon})}{(\omega+\theta_{\varepsilon})^{2}} |\nabla \theta_{\varepsilon}|^{2}) dx dt \\
& \leq \int_{0}^{T} \int_{\Omega} (\frac{\theta_{\varepsilon}}{\omega+\theta_{\varepsilon}} \rho_{\varepsilon}  {\rm div} u_{\varepsilon} dx dt +
\varepsilon \int_{0}^{T} \int_{\Omega} \theta^{\alpha}_{\varepsilon} dx dt \\
& - \int_{\Omega} (\rho_{0,\varepsilon}+ \varepsilon) Q_{h,\omega} (\theta_{0,\varepsilon}) dx + \int_{\Omega} (\rho_{\varepsilon}+ \varepsilon) Q_{h,\omega} (\theta_{\varepsilon})(T-) dx,
\end{aligned}
\end{equation}
where
\begin{equation*}
Q_{h,\omega} (\theta)= \int_{1}^{\theta} \frac{1}{\omega+z} dz ,
\end{equation*}

Letting $\omega \rightarrow 0$ and taking hypothesis (1.38) together with the estimate \eqref{3.16} into account, we have
\begin{equation}\label{3.23}
\begin{aligned}
&\int_{0}^{T} \int_{\Omega} (\frac{1}{1+\theta_{\varepsilon}} \mathbb{S}_{\varepsilon}: \nabla u_{\varepsilon} +  |\nabla \theta|^{2} + |\nabla \theta_{\varepsilon}^{\alpha/2}|^{2} ) dx dt \\
& \leq C(1+\int_{0}^{T} \int_{\Omega}  \rho_{\varepsilon}  {\rm div} u_{\varepsilon} dx dt ) \leq C,
\end{aligned}
\end{equation}

By virtue of Lemma 3.2 and above estimate, we know that
\begin{equation}\label{3.24}
\begin{aligned}
&\theta_{\varepsilon} ~~bounded~in~L^{2}(0,T; W^{1,2}(\Omega)),\\
&\theta_{\varepsilon}^{\alpha/2} ~~bounded~in~L^{2}(0,T; W^{1,2}(\Omega)),
\end{aligned}
\end{equation}

Similarly,taking
\begin{equation*}
\psi(t,x)= \varphi(t),~~~0\leq \varphi \leq 1,~~~\varphi \in \mathcal{D}(0,T),~~h(\theta)= \frac{1}{(1+\theta)^{\omega}},~~0<\omega<1,
\end{equation*}
in \eqref{2.53}, we can also deduce that
\begin{equation}\label{3.25}
\begin{aligned}
&\int_{0}^{T} \int_{\Omega} (\frac{1}{(1+\theta_{\varepsilon})^{\omega}} \mathbb{S}_{\varepsilon}: \nabla u_{\varepsilon} + \omega \frac{\kappa(\theta_{\varepsilon})}{(1+\theta_{\varepsilon})^{1+\omega}}  |\nabla \theta_{\varepsilon}|^{2} ) dx dt \\
& \leq C(1+\int_{0}^{T} \int_{\Omega}  \rho_{\varepsilon} \theta_{\varepsilon} {\rm div} u_{\varepsilon} dx dt ) \leq C,
\end{aligned}
\end{equation}
where $C$ is independent of both $\delta$ and $\omega$, which yields
\begin{equation*}
\|\theta_{\varepsilon}^{(\alpha+1-\omega)/2} \|_{L^{2}(0,T; W^{1,2}(\Omega))} \leq C(\omega)~~~for~~any~~\omega>0.
\end{equation*}

Finally, using Holder's inequality as in Section 5.2 \cite{Firesel} we establish that
\begin{equation}\label{3.26}
\int_{\{ \rho_{\varepsilon}>\omega \}}  \theta_{\varepsilon}^{\alpha+1} dx dt \leq C(\omega)~~~for~~any~~\omega>0.
\end{equation}

Since the density $\rho_{\varepsilon}$ solves the mass equation  in $\mathcal{D}^{\prime}((0,T))$, the total mass $M_{\varepsilon}$ is a constant of motion, and we have
\begin{equation}\label{3.27}
\int_{\rho_{\varepsilon}>\omega}  \rho_{\varepsilon} dx \geq M_{\varepsilon} -\omega |\Omega|\geq \frac{M}{2} - \omega |\Omega|,
\end{equation}

On the other hand, a straightforward application of Holder inequality gives rise to
\begin{equation}\label{3.28}
\int_{\{ \rho_{\varepsilon}>\omega \}}  \rho_{\varepsilon} dx \leq |\{\rho_{\varepsilon} \geq \omega \}|^{2/3} \| \rho_{\varepsilon} \|_{L^{3}(\Omega)}.
\end{equation}

Consequently, by virtue of \eqref{3.27}, \eqref{3.28}, \eqref{3.12} there exists a function $d=d(\omega)$, which is independent of $\varepsilon$, such that '
\begin{equation}\label{3.29}
|\{ \rho_{\varepsilon}>\omega \} | \geq d(\omega) >0~~for~all~t\in[0,T]~provided ~0\leq \omega < \frac{M}{2|\Omega|}.
\end{equation}

Fix $0< \omega < M/4|\Omega|$ and find a function $B\in C^{\infty}(R)$ such that
\begin{equation*}
B: R \rightarrow ~non-increasing, B(z)=0 ~~for~~ z\leq \omega,~~B(z)=-1~~for ~z\geq 2\omega.
\end{equation*}

For each $t\in [0,T]$, let $\eta= \eta_{\varepsilon}$ be the unique strong solution of the Neumann problem
\begin{equation}\label{3.30}
\begin{aligned}
&\Delta \eta_{\varepsilon} = B(\rho_{\varepsilon}(t))- \frac{1}{|\Omega|} \int_{\Omega}  B(\rho_{\varepsilon}(t)) dx ~~in~\Omega,\\
&\nabla \eta_{\varepsilon} \cdot n=0~~~on ~\partial\Omega,\\
&\int_{\Omega} \eta_{\varepsilon} dx=0,
\end{aligned}
\end{equation}

Since the right-hand side of \eqref{3.30} is uniformly bounded independently of $\varepsilon$, there is a constant $\underline{\eta}$ such that
\begin{equation*}
\eta_{\varepsilon} \geq \underline{\eta} ~~for~~all ~~t\in [0,T], x\in \Omega,\delta>0,
\end{equation*}

Accordingly, we can take a test function
\begin{equation*}
\varphi(t,x)\equiv \psi(t) (\eta_{\varepsilon}(t,x)- \underline{\eta}), \psi \in \mathcal{D}(0,T), ~~0\leq \psi\leq 1
\end{equation*}
in \eqref{2.53} to deduce
\begin{equation}\label{3.31}
\begin{aligned}
&\int_{0}^{T} \int_{\Omega} \psi \mathcal{K}_{h}(\theta_{\varepsilon}) (B(\rho_{\varepsilon}) - \frac{1}{|\Omega|} \int_{\Omega}  B(\rho_{\varepsilon}(t)) dx ) dx dt\\
& \leq 2 \| \eta_{\varepsilon}\|_{L^{\infty}((0,T)\times \Omega)} (\int_{0}^{T} \int_{\Omega} \varepsilon \theta_{\varepsilon}^{\alpha+1} + \theta_{\varepsilon} \rho_{\varepsilon} |{\rm div} u_{\varepsilon}| dx dt)\\
& + \| \nabla \eta_{\varepsilon}\|_{L^{\infty}((0,T)\times \Omega)} \int_{0}^{T} \int_{\Omega} \rho_{\varepsilon} Q_{h}(\theta_{\varepsilon}) |u_{\varepsilon}| dx dt\\
& + \int_{0}^{T} \int_{\Omega}  (\rho_{\varepsilon} + \varepsilon) Q_{h}(\theta_{\varepsilon})  (\underline{\eta} - \eta_{\varepsilon}) \partial_{t} \psi - (\rho_{\varepsilon} + \varepsilon) Q_{h}(\theta_{\varepsilon})  \partial_{t} \eta_{\varepsilon}  \psi dx dt,
\end{aligned}
\end{equation}

Now we can take a sequence of function $h=h_{n} \nearrow 1$ so that \eqref{3.31} gives rise to
\begin{equation}\label{3.32}
\begin{aligned}
&\int_{0}^{T} \int_{\Omega} \psi \mathcal{K}(\theta_{\varepsilon}) (B(\rho_{\varepsilon}) - \frac{1}{|\Omega|} \int_{\Omega}  B(\rho_{\varepsilon}(t)) dx ) dx dt\\
& \leq C(1+\int_{0}^{T} \int_{\Omega} (\rho_{\varepsilon}+ \varepsilon)\theta_{\varepsilon} |\partial_{t}\eta| dx dt ).
\end{aligned}
\end{equation}
Moreover,
\begin{equation*}
\begin{aligned}
&\int_{0}^{T} \int_{\Omega} \psi \mathcal{K}(\theta_{\varepsilon}) (B(\rho_{\varepsilon}) - \frac{1}{|\Omega|} \int_{\Omega}  B(\rho_{\varepsilon}(t)) dx ) dx dt\\
& = \int_{\{\rho_{\varepsilon}< \omega\}} \psi \mathcal{K}(\theta_{\varepsilon}) (B(\rho_{\varepsilon}) - \frac{1}{|\Omega|} \int_{\Omega}  B(\rho_{\varepsilon}(t)) dx ) dx dt \\
&+ \int_{\{\rho_{\varepsilon}\geq \omega\}} \psi \mathcal{K}(\theta_{\varepsilon}) (B(\rho_{\varepsilon}) - \frac{1}{|\Omega|} \int_{\Omega}  B(\rho_{\varepsilon}(t)) dx ) dx dt,
\end{aligned}
\end{equation*}
where, by virtue of \eqref{3.26}, the second integral on the right-hand side is bounded in dependent of $\varepsilon>0$.

On the other hand,
\begin{equation*}
 - \frac{1}{|\Omega|} \int_{\Omega}  B(\rho_{\varepsilon}(t)) dx ) dx  \geq
 - \frac{1}{|\Omega|} \int_{\rho_{\varepsilon} \geq 2\omega}  B(\rho_{\varepsilon}(t)) dx ) dx =
 \frac{|\rho_{\varepsilon} \geq 2\omega|}{|\Omega|} \geq \frac{d(2\omega)}{|\Omega|},
\end{equation*}
where we have used \eqref{3.29}. Thus we get
\begin{equation}\label{3.33}
\begin{aligned}
\int_{\{\rho_{\varepsilon}< \omega\}} \psi \mathcal{K}(\theta_{\varepsilon}) (B(\rho_{\varepsilon}) - \frac{1}{|\Omega|} \int_{\Omega}  B(\rho_{\varepsilon}(t)) dx ) dx dt \geq \frac{d(2\omega)}{|\Omega|} \int_{\{\rho_{\varepsilon}< \omega\}} \psi \mathcal{K}(\theta_{\varepsilon}) dx dt,
\end{aligned}
\end{equation}

This inequality, together with \eqref{3.32}, yields
\begin{equation}\label{3.34}
\begin{aligned}
\int_{0}^{T} \int_{_{\{\rho_{\varepsilon}< \omega\}}} \mathcal{K}(\theta_{\varepsilon})  dx dt\leq C(1+\int_{0}^{T} \int_{\Omega} (\rho_{\varepsilon}+ \varepsilon)\theta_{\varepsilon} |\partial_{t}\eta| dx dt ).
\end{aligned}
\end{equation}
Thus, the desired estimates on $\theta_{\varepsilon}$ in the space $L^{\alpha+1}((0,T)\times\Omega)$ provided we show that the integrals on the right hand side of \eqref{3.34} are bounded.

To this end, we use the fact that $\rho_{\delta}$ is a solution of the renormalized continuity equation and, consequently,
\begin{equation*}
\begin{aligned}
&\Delta \partial_{t} \eta_{\varepsilon} = \partial_{t} (\Delta \eta_{\varepsilon}) = \partial_{t} B(\rho_{\varepsilon}) - \frac{1}{|\Omega|}
\int_{\Omega} \partial_{t} B(\rho_{\varepsilon}) dx \\
& = - {\rm div} (B(\rho_{\varepsilon}) u_{\varepsilon}) - b(\rho_{\varepsilon}) {\rm div} u_{\varepsilon}
 + \frac{1}{|\Omega|} b(\rho_{\varepsilon}) {\rm div} u_{\varepsilon}dx,
\end{aligned}
\end{equation*}
whence
\begin{equation*}
\partial_{t} \eta ~~~in~~~L^{2}(0,T; W^{1,2}(\Omega)),
\end{equation*}
which, together with \eqref{3.12}, \eqref{3.24}, yields boundedness of the integrals on the right hand side of \eqref{3.34}.

Thus, we have shown that
\begin{equation}\label{3.35}
\theta_{\varepsilon} ~~is~bounded ~~in~~~L^{\alpha+1}((0,T)\times \Omega),
\end{equation}
by a constant which is independent of $\varepsilon>0$.

\subsection{Strict positivity of the temperature}
It is easy to see that inequality \eqref{2.53} holds also for functions
\begin{equation}\label{3.36}
\begin{aligned}
&h(\theta)= \frac{1}{\omega+ \theta},~~~\omega>0,\\
& \varphi(t,x)=\psi(t), ~~0\leq \psi \leq 1,~~\psi(0)=1,~~\psi(T)=1,~~\psi\in C^{\infty}[0,T].
\end{aligned}
\end{equation}

According, in view of the estimates obtained above, we have
\begin{equation}\label{3.37}
\begin{aligned}
&\int_{0}^{T} \int_{\Omega} (\varepsilon+ \rho_{\varepsilon}) Q_{h,\varepsilon}(\theta_{\varepsilon}) \partial_{t} \psi + \frac{k_{1}}{(\omega+\theta_{\varepsilon})^{2}} |\nabla \theta_{\varepsilon}|^{2} \psi dx dt\\
& \leq C- \int_{\Omega} (\varepsilon+ \rho_{0,\varepsilon}) Q_{h,\varepsilon}(\theta_{0,\varepsilon})dx,
\end{aligned}
\end{equation}
where
\begin{equation*}
Q_{h,\varepsilon}(\theta_{\varepsilon})\equiv \int_{1}^{\theta} \frac{C_{v}(z)}{(\omega+z)} dz,
\end{equation*}

Letting $\omega \rightarrow 0$ we can  conclude that
\begin{equation*}
\log(\theta_{\varepsilon}) ~~is~~bounded~~in~~L^{2}((0,T)\times \Omega)
\end{equation*}
by a constant independent of $\varepsilon>0$.

\subsection{Passing to the limits as $\varepsilon \rightarrow 0$.}

Applying the Aubin-Lions lemma and Lemma 3.3, we conclude
\begin{equation}\label{3.38}
\sqrt{\rho_{\varepsilon}} \rightarrow \sqrt{\rho}~~strongly~in~L^{2}(0,T; H^{1}(\Omega)),~~weakly~in~L^{2}(0,T; H^{2}(\Omega)),
\end{equation}
and
\begin{equation}\label{3.39}
\rho_{\varepsilon} u_{\varepsilon} \rightarrow \rho u~~strongly~in~L^{2}((0,T)\times \Omega),
\end{equation}
we notice that $u_{\varepsilon} \in L^{2}((0,T)\times \Omega)$, thus
\begin{equation}\label{3.40}
u_{\varepsilon} \rightarrow u~~weakly~in~L^{2}((0,T)\times \Omega),
\end{equation}
Thus we can pass to the limits for the term $\rho_{\varepsilon} u_{\varepsilon} \otimes u_{\varepsilon}$ as follows,
\begin{equation}\label{3.41}
\rho_{\varepsilon} u_{\varepsilon}\otimes u_{\varepsilon} \rightarrow \rho u\otimes u
\end{equation}
in  the distribution sense.

We can show
\begin{equation}\label{3.42}
\rho_{\varepsilon} |u_{\varepsilon}|^{2} u_{\varepsilon} \rightarrow \rho|u|^{2} u ~~~strongly~~in~~L^{1}((0,T)\times \Omega),
\end{equation}
similarly to Lemma 2.7.

By the previous estimates, for any test function $\psi\in L^{\infty}(0,T; L^{\infty}(\Omega))$  we can deduce that
\begin{equation}\label{3.43}
\varepsilon \int_{0}^{T} \int_{\Omega} \Delta \rho_{\varepsilon} \psi
\leq \varepsilon \|\Delta \rho_{\varepsilon} \|_{L^{2}(0,T; L^{2}(\Omega))} \|\psi\|_{L^{2}(0,T; L^{2}(\Omega))}   \rightarrow 0 ~~as~~in~~\varepsilon \rightarrow 0,
\end{equation}
and
\begin{equation}\label{3.44}
\begin{aligned}
\varepsilon \int_{0}^{T} \int_{\Omega}  \nabla \rho_{\varepsilon} \nabla u_{\varepsilon} \psi &\leq \varepsilon
\|\nabla \sqrt{\rho_{\varepsilon}} \|_{L^{2}(0,T; L^{2}(\Omega))} \|\sqrt{\rho_{\varepsilon}} u_{\varepsilon}\|_{L^{2}(0,T; L^{2}(\Omega))} \\
&\| \psi \|_{L^{\infty}(0,T; L^{\infty}(\Omega))} \rightarrow 0 ~~as~~in~~\varepsilon \rightarrow 0,
\end{aligned}
\end{equation}

For the convergence of term $\varepsilon \Delta^{2} u_{\mu}$, for any test function $\psi\in L^{2}(0,T: H^{2}(\Omega))$, thanks to \eqref{3.7}, we have
\begin{equation}\label{3.45}
| \int_{0}^{T} \int_{\Omega} \varepsilon \Delta^{2} u_{\varepsilon} \psi dx dt| \leq \sqrt{\varepsilon} \| \sqrt{\varepsilon} \Delta u_{\varepsilon}\|_{L^{2}((0,T)\times \Omega)} \| \Delta \psi \|_{L^{2}((0,T)\times \Omega)} \rightarrow 0~~as~~~~\varepsilon \rightarrow 0,
\end{equation}

For the convergence of terms $\varepsilon\rho_{\varepsilon}^{-10}$ and $\varepsilon\rho_{\varepsilon} \nabla \Delta^{9} \rho_{\varepsilon}$, we refer to the lemma 3.6 and 3.7 in \cite{Yu}.

Thus, by the compactness argument, we can pass to the limits as $\varepsilon \rightarrow 0$, yield that the limit function $(\rho,u,\theta)$ satisfy the continuity equation as well as the momentum equation:
\begin{equation}\label{3.46}
\partial_{t}\rho+{\rm div}(\rho u) =0,~~~pointwise~~~in~~~(0,T)\times \Omega.
\end{equation}
and
\begin{equation}\label{3.47}
\begin{aligned}
&(\rho u)_{t} + {\rm div} (\rho u\otimes u) + \nabla P - 2 {\rm div}(\rho \mathbb{D} u) + r_{0}u+ r_{1} \rho |u|^{2}
=\kappa \rho \nabla (\frac{\Delta \sqrt{\rho}}{\sqrt{\rho}}), \\ &~~~holds~in~the~sense~of~distribution~on~(0,T)\times \Omega,
\end{aligned}
\end{equation}
Furthermore, thanks to the weak lower semicontinuity of convex functions, we are able to pass to the limits in the energy inequality \eqref{2.12} and B-D entropy inequality \eqref{3.5}; by the strong convergence of the density and temperature, we have the following energy inequality in the sense of distributions on $(0,T)$:
\begin{equation}\label{3.48}
\begin{aligned}
&\sup_{t\in (0,T)} \int_{\Omega} E(\rho,u,\theta) + r_{0} \int_{\Omega} |u|^{2} dx + r_{1} \int_{\Omega} \rho |u|^{4}dx  +   \kappa \varepsilon  \int_{\Omega} \rho |\nabla^{2} \log \rho|^{2}dx \leq E_{0}(\rho,u,\theta),
\end{aligned}
\end{equation}
where
\begin{equation}\label{3.49}
E(\rho,u,\theta)= \int_{\Omega} (\frac{1}{2} \rho |u|^{2} + \frac{\eta}{10} \rho^{-10}+ \frac{\kappa}{2}|\nabla \sqrt{\rho}|^{2} + \frac{\delta}{2}
|\nabla \Delta^{4} \rho|^{2}+ \rho\theta + \beta \theta^{4})dx,
\end{equation}
and
\begin{equation}\label{3.50}
\begin{aligned}
&\int_{\Omega} (\frac{1}{2} \rho |u+ \frac{\nabla \rho}{\rho}|^{2} + \frac{\kappa}{2}|\nabla \sqrt{\rho}|^{2} -r_{0} \log \rho_{\varepsilon})dx
+ \kappa  \int_{0}^{T} \int_{\Omega} \rho |\nabla^{2} \log \rho|^{2}dx\\
&+\frac{1}{2} \int_{0}^{T}\int_{\Omega} \rho |\nabla u- \nabla^{T} u|^{2}dx + \int_{\Omega} \frac{|\nabla \rho|^{2}}{\rho} \theta\leq
\int_{\Omega} (\frac{1}{2} \rho_{0} |u_{0}+ \frac{\nabla \rho_{0}}{\rho_{0}}|^{2} + \frac{\kappa}{2}|\nabla \sqrt{\rho_{0}}|^{2}  + 2E_{0},
\end{aligned}
\end{equation}

\subsection{Thermal energy equation}
In order to complete the limit passage $\varepsilon \rightarrow 0$, we have to show that $\rho, u$ and $\theta$ represent a variational solution of the thermal energy equation (1.36) in the sense of Definition 1.1.

Because of the uniform estimate \eqref{3.23}, we know that
\begin{equation}\label{3.51}
Q(\theta_{\varepsilon}) \rightarrow  \overline{Q(\theta)} ~~weakly ~~~in~~~L^{2}(0,T; W^{1,2}(\Omega)).
\end{equation}
which, together with \eqref{3.38}, yields
\begin{equation*}
\rho_{\varepsilon}Q(\theta_{\varepsilon}) \rightarrow  \rho \overline{Q(\theta)} ~~weakly ~~~in~~~L^{2}(0,T; L^{2}(\Omega)).
\end{equation*}

Thus, we are allowed to apply to Lemma 6.3 to de deduce
\begin{equation*}
\rho_{\varepsilon} Q(\theta_{\varepsilon}) \rightarrow \rho \overline{Q(\theta)} ~~~in~~~L^{2}(0,T; W^{-1,2}(\Omega)).
\end{equation*}
In accordance with \eqref{3.51}, we have
\begin{equation*}
\rho_{\varepsilon} Q(\theta_{\varepsilon})^{2} \rightarrow \rho \overline{Q(\theta)}^{2} ~~~in~~~\mathcal{D}^{\prime}((0,T)\times \Omega).
\end{equation*}
Since $Q$ is sublinear, we can infer that
\begin{equation*}
\theta_{\varepsilon} \rightarrow \overline{\theta} ~~~(strongly) ~~in ~~ L^{r}(\{ \rho>0 \}) ~~for ~~a~~certain ~~r>1.
\end{equation*}

Now  we can pass to limits for $\varepsilon \rightarrow 0$ in \eqref{2.53} to obtain
\begin{equation}\label{3.52}
\begin{aligned}
&\int_{0}^{T} \int_{\Omega}  \rho Q_{h}(\overline{\theta})  \partial_{t}  \varphi + \rho_{\varepsilon}\rho Q_{h}(\overline{\theta}) u\cdot \nabla \varphi +  \overline{\mathcal{K}_{h}(\overline{\theta})}  \Delta \varphi\\
& \int_{0}^{T} \int_{\Omega} (h(\overline{\theta}) \mathbb{S}:\nabla u+ h(\overline{\theta}) \overline{\theta} \rho {\rm div}u) \varphi dx dt
- \int_{\Omega} \rho_{0} Q_{h}(\theta_{0})   \varphi_{0} dx,
\end{aligned}
\end{equation}
where
\begin{equation*}
\rho \overline{\mathcal{K}_{h}(\theta)}=\rho \mathcal{K}_{h}(\overline{\theta}),
\end{equation*}
and
\begin{equation*}
\log (\overline{\mathcal{K}_{h}(\theta)})\in  L^{2}((0,T)\times \Omega).
\end{equation*}

Take
\begin{equation*}
h(\theta)= \frac{1}{(1+\theta)^{\omega}}, ~~0<\omega<1,
\end{equation*}
in \eqref{3.52} and let $\omega \rightarrow 0$  in order to deduce
\begin{equation}
\begin{aligned}
&\int_{0}^{T} \int_{\Omega}  \rho \overline{\theta}  \partial_{t}  \varphi + \rho_{\varepsilon}\rho \overline{\theta} u\cdot \nabla \varphi +  \overline{\mathcal{K}(\theta)}  \Delta \varphi\leq \\
& \int_{0}^{T} \int_{\Omega} ( \mathbb{S}:\nabla u+ \overline{\theta} \rho {\rm div}u) \varphi dx dt
- \int_{\Omega} \rho_{0} \theta_{0}   \varphi_{0} dx,
\end{aligned}
\end{equation}

Finally, we set
\begin{equation*}
\theta \equiv \mathcal{K}^{-1} (\overline{\mathcal{K}(\theta)}).
\end{equation*}
Obviously, the new function $\theta$ is non-negative, specifically,
\begin{equation*}
\theta\in L^{\alpha+1}((0,T)\times  \Omega),~~ \log(\theta) \in  L^{2}((0,T)\times  \Omega),
\end{equation*}

Therefore we obtain a variational form of the thermal energy inequality:
\begin{equation}\label{1.15}
\begin{aligned}
&\int_{0}^{T} \int_{\Omega} \rho \theta \partial_{t} \phi + \rho \theta \cdot \phi + \mathcal {K} (\theta) \Delta \phi dx dt \leq \\
&\int_{0}^{T} \int_{\Omega} (R \rho \theta - \mathbb{S}: \nabla u ) \phi dx dt - \int_{\Omega} \rho_{0} \theta_{0} dx,
\end{aligned}
\end{equation}
to be satisfied for any test function
\begin{equation}
\phi\in C^{\infty}([0,T]\times \Omega),~~~\phi \geq 0,~~~\phi(0)=1~~~\phi(T)=0,
\end{equation}

\section{Approximation of the Mellet-Vasseur type inequality}
As seen before, we can deduce the strong compactness of the density and temperature from the B-D energy estimate and entropy estimate. Note that estimates are independent of all approximation parameter. Unfortunately, the primary obstacle to prove the compactness of the solution to (1.1) is the lack of strong convergence for $\sqrt{\rho}u$ in $L^{2}$. To solve this problem, a new estimate as established in Mellet and Vasseur \cite{Mellet}, providing a $L^{\infty}(0,T; L\log L(\Omega))$ control on $\rho|u|^{2}$. This new estimate enable us to pass to the limit $r_{0}\rightarrow 0, r_{1} \rightarrow 0$ and $\kappa \rightarrow 0$.

In this section, we construct an approximation of the Mellet-Vasseur type inequality for any weak solutions to the following level of approximate system:
\begin{gather}
  \partial_{t}\rho+{\rm div}(\rho u) =0,\label{4.1} \\
  (\rho u)_{t} + {\rm div} (\rho u\otimes u) + \nabla P - 2 {\rm div}(\rho \mathbb{D} u) + r_{0}u+ r_{1} \rho |u|^{2}
=\kappa \rho \nabla (\frac{\Delta \sqrt{\rho}}{\sqrt{\rho}}), \label{4.2}\\
  \partial_{t} ( \rho \theta+ \beta \theta^{4}) + {\rm div} ( u(\rho \theta+ \beta \theta^{4})) + {\rm div} q = \rho |\nabla u|^{2}-P {\rm div} u,\label{4.3}
\end{gather}

Following the idea in \cite{Vesseur}, we define two $C^{\infty}$, nonnegative cut-off function $\phi_{m}$ and $\phi_{K}$ as follows:
\begin{equation}
\phi_{m}(\rho)=1~~~~for ~any~\rho>\frac{1}{m},~~~\phi_{m}(\rho)=0~~~~for~any~\rho<\frac{1}{2m},
\end{equation}
where $m>0$ is any real number, and $|\phi^{\prime}_{m}|\leq 2m$; and $\phi_{K}(\rho)\in C^{\infty}(\mathbb{R})$ is a nonnegative function such that
\begin{equation}
\phi_{K}(\rho)=1~~~for~any~\rho<K,~~\phi_{K}(\rho)=0~~~~for~any~\rho>2K,
\end{equation}
where $K>0$ is any real number, and $|\phi^{\prime}_{K}|\leq \frac{2}{K}$.

We define $v=\phi(\rho) u$, and $\phi(\rho)= \phi_{m}(\rho) \phi_{K}(\rho)$. The following lemma will be useful to construct the approximation of the Mellect-Vasseur type inequality. The structure of the $\kappa$ quantum term is essential to get this lemma in 3D.

\begin{lemm}
For any fixed $\kappa>0$, we have
\begin{equation}
\|\nabla v\|_{L^{2}()0,T; L^{2}(\Omega)} \leq C
\end{equation}
where the constant C depend on $\kappa>0$, $r_{1}$, $K$ and $m$; and
\begin{equation}
\rho_{t} \in L^{4}(0,T; L^{\frac{6}{5}}(\Omega)) + L^{2}(0,T; L^{\frac{3}{2}}(\Omega))~~~uniform ~~in~~\kappa.
\end{equation}

\end{lemm}

We introduce a new nonnegative cut-off function $\varphi_{n}$ which is in $C^{1}(R^{3})$:
\begin{equation}
\varphi_{n}(u) = \tilde{\varphi}_{n}(|u|^{2}),
\end{equation}
where $\tilde{\varphi}_{n}$ is given on $R^{+}$ by
$$ \varphi_{n}^{\prime \prime}(y)=\left\{
\begin{aligned}
&= \frac{1}{1+y}~~~~~~~~~~~~if~~0\leq y\leq n,\\
&= -\frac{1}{1+y}~~~~~~~~~~~~if~~n\leq y\leq C_{n},\\
&=0~~~~~~~~~~~~~~~~if ~~y\geq C_{n},
\end{aligned}
\right.
$$
with $\varphi_{n}^{\prime}(0)=0, \varphi_{n}(0)=0$, and $C_{n}= e(1+n)^{2}-1$.

Here we gather the properties of the function $\varphi_{n}^{\prime}$ in the following lemma:
\begin{lemm}
Let $\varphi_{n}$ and $\tilde{\varphi}_{n}$ be defined as above. Then they verify

$\bullet$  (a) For any $u\in \mathbb{R}^{3}$, we have
\begin{equation}
\varphi^{\prime \prime}_{n}(u) = 2(2\tilde{\varphi}^{\prime \prime}_{n}(|u|^{2}) u \otimes u + I
\tilde{\varphi}^{\prime}_{n}(|u|^{2})),
\end{equation}
where $I$ is $3\times 3$ identity matrix.

$\bullet$ (b) $\varphi^{\prime \prime}_{n}(y)\leq \frac{1}{1+y}$ for any $n>0$ and $y\geq 0$.

$\bullet$ (c)
$$ \varphi_{n}^{\prime}(y)\left\{
\begin{aligned}
&= 1+ \ln(1+y)~~~~~~~~~~~~if~~0\leq y\leq n,\\
&= 0~~~~~~~~~~~~~~~~~~~~~~~~~~~~~~~~~~~~~~if ~~y\geq C_{n},\\
&\geq 0,~~and~~\leq 1+\ln(1+y)~~~~if~~n\leq y\leq C_{n},
\end{aligned}
\right.
$$
In one word, $0 \leq \varphi_{n}^{\prime}\leq 1+ \ln(1+y)$ for any $y\geq 0$, and it is compactly supported.

$\bullet$ (d)
For any given $n>0$, we have
\begin{equation}
|\varphi_{n}^{\prime \prime}(u)| \leq 6 + 2 \ln (1+n)
\end{equation}
for any $u \in \mathbb{R}^{3}$.

$\bullet$ (e)
$$ \tilde{\varphi}_{n}(y)\left\{
\begin{aligned}
&= (1+y) \ln(1+y)~~~~~~~~~~~~if~~0\leq y\leq n,\\
&= 2(1+ \ln(1+n))y- (1+y) \ln(1+y) + 2(\ln(1+n)-n) ~~~~if~~n\leq y\leq C_{n},\\
&= e(1+n)^{2}-2n-2~~~~~~~~~~~~~~~~~~~~~~~~~~if ~~y\geq C_{n},
\end{aligned}
\right.
$$
$\tilde{\varphi}_{n}(y)$ is a nondecreasing function with respect to y for any fixed n, and it is a nondecreasing function with respect to n for any fixed y, and
\begin{equation}
\tilde{\varphi}_{n}(y)\rightarrow (1+y) \ln (1+y) ~~~a.e.
\end{equation}
as $n\rightarrow \infty$.
\end{lemm}

By the molifier method, we can  construct the approximation of the Mellet-Vasseur type inequality which  is  shown in the following lemma:

\begin{lemm}
For any weak solution to \eqref{4.1}-\eqref{4.3}, and any $\psi\in \mathfrak{D}(-1,+\infty)$, we have
\begin{equation}\label{4.12}
\begin{aligned}
&-\int_{0}^{T} \int_{\Omega} \psi_{t} \rho \varphi_{n}(v)
 dx dt + \int_{0}^{T} \int_{\Omega} \psi(t) \varphi^{\prime}_{n}(v) F dx dt\\
& + \int_{0}^{T} \int_{\Omega} \psi(t) \mathbb{S}: \nabla(\varphi^{\prime}_{n}(v)) dx dt\\
& = \int_{\Omega} \rho_{0} \varphi_{n}(v_{0}) \psi(0)dx,
\end{aligned}
\end{equation}
where
\begin{equation}
\mathbb{S} = \rho \varphi(\phi) (\mathbb{D} u + \kappa \frac{\Delta \sqrt{\rho}}{\sqrt{\rho}}\mathbb{I}), ~~~~and
\end{equation}
\begin{equation}
F= \rho^{2} u \phi_{K}^{\prime} (\rho) {\rm div} u + \nabla P \phi_{K}(\rho) + r_{0} u \phi_{K}(\rho)+ r_{1} \rho |u|^{2} u \phi_{K}(\rho)+ \kappa \sqrt{\rho} \nabla \phi_{K}(\rho) \Delta \sqrt{\rho} + 2\kappa \phi_{K}(\rho) \nabla \sqrt{\rho} \Delta \sqrt{\rho}),
\end{equation}
where $\mathbb{I}$ is an identical matrix.
\end{lemm}

\section{Recover the limits as $m\rightarrow \infty$ and $K\rightarrow \infty $}
In this section, we want to recover the limits from \eqref{4.12} as $m\rightarrow \infty$ and $K\rightarrow \infty $. Firstly, we will pass to the limit $m\rightarrow 0$. For the $K\rightarrow \infty$case, it is similar to the $m\rightarrow 0$ process.  For any fixed weak solution $(\rho, u)$, $\phi_{m}(\rho)$ converges to 1 almost everywhere for $(t,x)$, and it is uniform bounded in $L^{\infty}(0,T; \Omega)$, and
\begin{equation}
r_{0} \phi_{K}(\rho) u \in L^{2}(0,T; L^{2}(\Omega)).
\end{equation}
Thus, we find
\begin{equation}
v_{m}=\phi_{m} \phi_{K}(\rho) u \in L^{2}(0,T; L^{2}(\Omega))\rightarrow \phi_{K} u ~~~almost ~~~everywhere ~~~for ~~~(t,x).
\end{equation}
as $m\rightarrow \infty$. The dominated convergence theorem allows us to have
\begin{equation}
v_{m} \rightarrow  \phi_{K} u ~~~in ~~~L^{2}(0,T; L^{2}(\Omega))
\end{equation}
as $m\rightarrow \infty$, and hence
\begin{equation}
\varphi_{n}(v_{m}) \rightarrow \varphi_{n}(\phi_{K} u)~~~~~in~~~L^{p}((0,T)\times \Omega)
\end{equation}
for any $1\leq p < \infty$. Thus, we can show that
\begin{equation}
\int_{0}^{T} \int_{\Omega} \psi^{\prime}(t) (\rho \varphi_{n}(v_{m})) dx dt \rightarrow
\int_{0}^{T} \int_{\Omega} \psi^{\prime}(t) (\rho \varphi_{n}(\phi_{K} u)) dx dt \rightarrow,
\end{equation}
and
\begin{equation}
\int_{0}^{T} \int_{\Omega} \rho \varphi_{n}(v_{m0})) dx dt \rightarrow
\int_{0}^{T} \int_{\Omega} \rho_{0} \varphi_{n}(\phi_{K}(\rho_{0}) u_{0})) dx dt \rightarrow,
\end{equation}
as $m\rightarrow \infty$.

Meanwhile, for any fixed $\rho$, we have
\begin{equation}
\phi^{\prime}_{m}(\rho) \rightarrow 0~~~~almost~~~everywhere~~for~~(t,x)
\end{equation}
as $m\rightarrow \infty$.

Calculating $|\phi^{\prime}_{m}(\rho)| \leq 2m$ as $\frac{1}{2m} \leq \rho \leq \frac{1}{m}$, and otherwise, $\phi^{\prime}_{m}(\rho)=0$, thus
\begin{equation}
|\rho \phi^{\prime}_{m}(\rho)| \leq 1~~~for~~all~~\rho.
\end{equation}

To pass into the limits in \eqref{4.12} as $m\rightarrow \infty$, we rely on the following Lemma:
\begin{lemm}
If
\begin{equation}
\begin{aligned}
&\| a_{m} \|_{L^{\infty}(0,T); \Omega} \leq C,~~~~~a_{m} \rightarrow a~~~~a.e.~~for ~~(t,x) \\
&~~~and~~~in~~L^{p}((0,T)\times \Omega)~~~~for~~any~~1\leq p <\infty.
\end{aligned}
\end{equation}
$f\in L^{1}((0,T)\times \Omega)$, then we have
\begin{equation}
\int_{0}^{T} \int_{\Omega} \phi_{m}(\rho) a_{m} f dx dt \rightarrow  \int_{0}^{T} \int_{\Omega} a f dx dt ~~~~~as~~~~m\rightarrow infty.
\end{equation}
\end{lemm}

Calculating
\begin{equation}\label{5.11}
\begin{aligned}
&\int_{0}^{T} \int_{\Omega} \psi(t) \mathbb{S}_{m} : \nabla (\varphi^{\prime}(v_{m})) dx dt\\
&\int_{0}^{T} \int_{\Omega} \psi(t) \mathbb{S}_{m}\varphi^{\prime\prime}(v_{m})(\nabla \phi_{m} \phi_{K} u +
\phi_{m} \nabla \phi_{K} u + \phi_{m}\phi_{K}+ \phi_{m}\phi_{K} \nabla u) dx dt\\
&\int_{0}^{T} \int_{\Omega} \phi_{m} a_{m1} f_{1} dx dt+
\int_{0}^{T} \int_{\Omega} \phi_{m} a_{m2} f_{2} dx dt,
\end{aligned}
\end{equation}
where
where
\begin{equation*}
a_{m1}= \phi_{m}(\rho)\varphi^{\prime \prime}(v_{m})
\end{equation*}
\begin{equation*}
f_{1}= \psi(t) \rho \phi_{K}(\rho)(\mathbb{D} u +\kappa\frac{\Delta \sqrt{\rho}}{\sqrt{\rho}} \mathbb{B})(u \nabla \phi_{K}+
\phi_{K}(\rho) \nabla u),
\end{equation*}
and
\begin{equation*}
a_{m2}= \phi_{m}(\rho)\phi_{\rho} u \varphi^{\prime \prime}(v_{m})= \varphi^{\prime \prime}(v_{m}) v_{m},
\end{equation*}
\begin{equation*}
\begin{aligned}
&f_{2}= \psi(t)  \phi_{K}(\rho)(\mathbb{D} u +\kappa\frac{\Delta \sqrt{\rho}}{\sqrt{\rho}} \mathbb{I})\nabla \rho \\
& = 2 \psi(t) \phi_{K} (\rho) (\kappa \Delta \rho \nabla \sqrt{\rho} + \sqrt{\rho} \mathbb{D} u \nabla \sqrt{\rho} ).
\end{aligned}
\end{equation*}
So applying Lemma 3.1 to \eqref{5.11}, one obtains
\begin{equation}
\int_{0}^{T} \int_{\Omega} \psi(t) \mathbb{S}_{m} : \nabla (\varphi^{\prime}(v_{m})) dx dt \rightarrow
\int_{0}^{T} \int_{\Omega} \psi(t) \mathbb{S}_{m} : \nabla (\varphi^{\prime}(\phi_{K}(\rho) u)) dx dt
\end{equation}
as $m \rightarrow \infty$, where $\mathbb{S}= \phi_{K}(\rho)(\mathbb{D} u +\kappa\frac{\Delta \sqrt{\rho}}{\sqrt{\rho}} \mathbb{I})$.

Letting $F_{m}= F_{m1}+ F_{m2}$, where
\begin{equation*}
\begin{aligned}
&F_{m1}= \rho^{2} u \phi^{\prime}(\rho) {\rm div} u + \rho \nabla \phi(\rho) \mathbb{D} u + \kappa \sqrt{\rho} \nabla \phi(\rho) \Delta \sqrt{\rho}\\
&= \rho (\phi^{\prime}_{m}(\rho) \phi_{K}(\rho) + \phi_{m}(\rho) \phi^{\prime}_{K}(\rho)) (\rho u {\rm div} u + \nabla \rho \cdot \mathbb{D}u + \kappa\nabla \rho \frac{\Delta \sqrt{\rho}}{\sqrt{\rho}}),
\end{aligned}
\end{equation*}
where
\begin{equation}
\phi_{K}(\rho)(\rho u {\rm div} u + \nabla \rho \cdot \mathbb{D}u + \kappa\nabla \rho \frac{\Delta \sqrt{\rho}}{\sqrt{\rho}}) \in L^{1}((0,T)\times \Omega),
\end{equation}
and
\begin{equation}
\rho \phi^{\prime}_{K}(\rho)(\rho u {\rm div} u + \nabla \rho \cdot \mathbb{D}u + \kappa\nabla \rho \frac{\Delta \sqrt{\rho}}{\sqrt{\rho}}) \in L^{1}((0,T)\times \Omega),
\end{equation}
and
\begin{equation}
F_{m2}= \phi_{m}(\rho) \phi_{K}(\rho) (2\rho^{\frac{\gamma}{2}} \nabla \rho^{\frac{\gamma}{2}} + r_{0} u + r_{1} \rho |u|^{2} u + 2\kappa \nabla \sqrt{\rho} \Delta \sqrt{\rho}),
\end{equation}
where
\begin{equation}
\phi_{K}(\rho) (2\rho^{\frac{\gamma}{2}} \nabla \rho^{\frac{\gamma}{2}} + r_{0} u + r_{1} \rho |u|^{2} u + 2\kappa \nabla \sqrt{\rho} \Delta \sqrt{\rho}) \in L^{1}((0,T)\times \Omega).
\end{equation}
 Using Lemma 3.1, we obtain
 \begin{equation}
\int_{0}^{T} \int_{\Omega} \psi(t) \varphi^{\prime}(v_{m}) F_{m} dx dt \rightarrow \int_{0}^{T} \int_{\Omega} \psi(t) \varphi^{\prime}(\phi_{K}(\rho) u) F_{m} dx dt,
\end{equation}
where
\begin{equation}
F= \rho^{2} u \phi_{K}^{\prime} (\rho) {\rm div} u + 2\rho^{\frac{\gamma}{2}} \nabla \rho^{\frac{\gamma}{2}} \phi_{K}(\rho) + r_{0} u \phi_{K}(\rho)+ r_{1} \rho |u|^{2} u \phi_{K}(\rho)+ \kappa \sqrt{\rho} \nabla \phi_{K}(\rho) \Delta \sqrt{\rho} + 2\kappa \phi_{K}(\rho) \nabla \sqrt{\rho} \Delta \sqrt{\rho}),
\end{equation}
Thus, letting $m\rightarrow \infty$ in \eqref{4.12}, and using the above convergence in this section, we find
\begin{equation}
\begin{aligned}
&- \int_{0}^{T} \int_{\Omega} \psi_{t} \rho \varphi_{n}(\phi_{K}(\rho) u) dx dt + \int_{0}^{T} \int_{\Omega} \psi_{\tau}(t) \varphi^{\prime}_{n}(\phi_{K}(\rho) u) F dx dt\\
& + \int_{0}^{T} \int_{\Omega} \psi_{\tau}(t) \mathbb{S}: \nabla (\varphi^{\prime}_{n}(\phi_{K}(\rho) u)) dx dt= \int_{\Omega} \rho_{0} \varphi_{n}(\phi_{K}(\rho_{0}) u_{0})dx dt,
\end{aligned}
\end{equation}
which in turn gives us the following lemma:
\begin{lemm}
For any weak solution to \eqref{4.1}-\eqref{4.3}, we have
\begin{equation}
\begin{aligned}
&- \int_{0}^{T} \int_{\Omega} \psi_{t} \rho \varphi_{n}(\phi_{K}(\rho) u) dx dt + \int_{0}^{T} \int_{\Omega} \psi_{\tau}(t) \varphi^{\prime}_{n}(\phi_{K}(\rho) u) F dx dt\\
& + \int_{0}^{T} \int_{\Omega} \psi_{\tau}(t) \mathbb{S}: \nabla (\varphi^{\prime}_{n}(\phi_{K}(\rho) u)) dx dt= \int_{\Omega} \rho_{0} \varphi_{n}(\phi_{K}(\rho_{0}) u_{0})dx dt,
\end{aligned}
\end{equation}
where $\mathbb{S}= \phi_{K}(\rho) \rho (\mathbb{D} u + \kappa \frac{\Delta \sqrt{\rho}}{\sqrt{\rho}} \mathbb{I}) $, and
\begin{equation}
F= \rho^{2} u \phi_{K}^{\prime} (\rho) {\rm div} u + \nabla P \phi_{K}(\rho) + r_{0} u \phi_{K}(\rho)+ r_{1} \rho |u|^{2} u \phi_{K}(\rho)+ \kappa \sqrt{\rho} \nabla \phi_{K}(\rho) \Delta \sqrt{\rho} + 2\kappa \phi_{K}(\rho) \nabla \sqrt{\rho} \Delta \sqrt{\rho},
\end{equation}
where $\mathbb{I}$ is an identical matrix.
\end{lemm}

Similar to the passage $m\rightarrow 0$, letting $K\rightarrow \infty$, we deduce the following lemma:
\begin{lemm}
For any weak solution to \eqref{4.1}-\eqref{4.3} , we have
\begin{equation}\label{5.22}
\begin{aligned}
&- \int_{0}^{T} \int_{\Omega} \psi_{t} \rho \varphi_{n}(u) dx dt + \int_{0}^{T} \int_{\Omega} \psi_{\tau}(t) \varphi^{\prime}_{n}(u) F dx dt\\
& + \int_{0}^{T} \int_{\Omega} \psi_{\tau}(t) \mathbb{S}: \nabla (\varphi^{\prime}_{n}(u)) dx dt= \int_{\Omega} \rho_{0} \varphi_{n}(u_{0})dx dt,
\end{aligned}
\end{equation}
where $\mathbb{S}= \rho (\mathbb{D} u + \kappa \frac{\Delta \sqrt{\rho}}{\sqrt{\rho}} \mathbb{I}) $, and
\begin{equation}
F=  \nabla P  + r_{0} u + r_{1} \rho |u|^{2} u + 2\kappa \nabla \sqrt{\rho} \Delta \sqrt{\rho},
\end{equation}
where $\mathbb{I}$ is an identical matrix.
\end{lemm}

\section{Recover the limits as $\kappa \rightarrow 0$.}
In this section, Our aim is to recover the limits in \eqref{5.22} as $\kappa \rightarrow 0$. First, we have the following lemma.
\begin{lemm}
Let $\kappa \rightarrow 0$ ,  we have
 for any fixed n,
\begin{equation}
\rho_{\kappa} \varphi_{n}(u_{\kappa}) \rightarrow \rho \varphi_{n}(u)~~~~~~~~~strongly ~~~in ~~L^{1}((0,T)\times \Omega),
\end{equation}
and
\begin{equation}
\rho_{\kappa} \theta_{\kappa}^{2} (1+ \tilde{\varphi}^{\prime}_{n}(|u_{\kappa}|^{2})) \rightarrow \rho \theta^{2} (1+ \tilde{\varphi}^{\prime}_{n}(|u|^{2}))~~~~~~~~~strongly ~~~in ~~L^{1}((0,T)\times \Omega),
\end{equation}
\end{lemm}

With this lemma in hand, we are ready to recover the limits in \eqref{5.22} as $\kappa \rightarrow 0$. We have the following lemma.
\begin{lemm}
Let $\kappa \rightarrow 0$, for any $\psi \geq 0$ and $\psi^{\prime} \leq 0$, we have
\begin{equation}
\begin{aligned}
&- \int_{0}^{T} \int_{\Omega} \psi^{\prime} \rho \varphi_{n}(u) dx dt \\
& \leq 8 \|\psi\|_{L^{\infty}} (\int_{\Omega} (\rho_{0}|u_{0}|^{2}) + \frac{\rho^{\gamma}_{0}}{\gamma-1}) + |\nabla \sqrt{\rho_{0}}|^{2} - r_{0} \log_{-} \rho_{0}) dx + 2E_{0} \\
& + C(\|\psi\|_{L^{\infty}}) \int_{0}^{T} (\int_{\Omega} (\rho \theta^{2})^{\frac{2}{2-\delta}}dx)^{\frac{2}{2-\delta}}
\times  (\int_{\Omega} (1+\tilde{\varphi}^{\prime}_{n}(|u|^{2}))^{\frac{2}{\delta}}dx) ^{\frac{\delta}{2}} dt,
\end{aligned}
\end{equation}
\end{lemm}

\begin{proof}
By use of Lemma 6.1, we can handle the first and forth term in \eqref{5.22} as follows, that is,
\begin{equation}
\int_{0}^{T} \int_{\Omega} \psi^{\prime}(t) (\rho_{\kappa} \varphi_{n}(u_{\kappa})) dx dt \rightarrow  \int_{0}^{T} \int_{\Omega} \psi^{\prime}(t) (\rho \varphi_{n}(u)) dx dt
\end{equation}
and
\begin{equation}
\psi(0) \int_{\Omega} \rho_{0} \varphi^{\prime}_{0} (u_{\kappa,0}) dx \rightarrow \psi(0) \int_{\Omega} \rho_{0} \varphi^{\prime} (u_{0}) dx
\end{equation}
as $\kappa \rightarrow 0$ .

On the other hand, for the second term in \eqref{5.22} 
\begin{equation}
\int_{0}^{T} \int_{\Omega} \psi(t) \varphi^{\prime}_{n} (u_{\kappa}) \cdot \nabla (\rho_{\kappa} \theta_{\kappa})dx dt
 = - \int_{0}^{T} \int_{\Omega} \psi(t) \rho_{\kappa} \theta_{\kappa}  \varphi^{\prime \prime}_{n} : \nabla u_{\kappa} dx dt=P,
\end{equation}

Thanks to Part b of Lemma 4.2, we have
\begin{equation}
\varphi^{\prime\prime}(u_{\kappa}) :\nabla u_{\kappa} = 4\tilde{\varphi}^{\prime\prime}_{n} (|v_{\kappa}|^{2}) \nabla u_{\kappa} : (u_{\kappa}\otimes u_{\kappa}) + 2 {\rm div}u_{\kappa} \tilde{\varphi}^{\prime}_{n}(|u_{\kappa}|^{2}).
\end{equation}
Using Part b of Lemma 4.2, we find that
\begin{equation}
\begin{aligned}
&|\tilde{\varphi}^{\prime\prime}_{n} (|u_{\kappa}|^{2}) \nabla u_{\kappa} : (u_{\kappa}\otimes u_{\kappa})|
\leq |\tilde{\varphi}^{\prime\prime}_{n} (|v_{\kappa}|^{2})| |\nabla u_{\kappa}| |v_{\kappa}|^{2}\\
& \leq |\nabla u_{\kappa}| \frac{|u_{\kappa}|^{2}}{1+ |u_{\kappa}|^{2}} \leq |\nabla u_{\kappa}|,
\end{aligned}
\end{equation}
where we denote $|\nabla u_{\kappa}|^{2} = \sum_{ij} |\partial_{i} u_{j}|^{2}$. Hence
\begin{equation}
\begin{aligned}
&|P| \leq 4 \int_{0}^{T} \int_{\Omega} \psi(t) |\rho_{\kappa} \theta_{\kappa}|  |\nabla u_{\kappa}| dx dt\\
& + 2 \int_{0}^{T} \int_{\Omega} \psi(t) \int_{\Omega}  |\tilde{\varphi}^{\prime}_{n} (|u_{\kappa}|^{2})| |\rho_{\kappa} \theta_{\kappa}|
|{\rm div} u_{\kappa}| dx dt\\
& \leq 4\|\psi\|_{L^{\infty}} \int_{0}^{T} \int_{\Omega} \rho_{\kappa} |\nabla u_{\kappa}|^{2} dx dt + C(\|\psi\|_{L^{\infty}}) \int_{0}^{T} \int_{\Omega} \rho_{\kappa} \theta_{\kappa}^{2} dx dt\\
& + 2 \int_{0}^{T} \int_{\Omega} \psi(t) \int_{\Omega}  |\tilde{\varphi}^{\prime}_{n} (|u_{\kappa}|^{2})| |\rho_{\kappa} \theta_{\kappa}|
|{\rm div} u_{\kappa}| dx dt,
\end{aligned}
\end{equation}
and the term
\begin{equation}
\begin{aligned}
&2 \int_{0}^{T} \int_{\Omega} \psi(t) \int_{\Omega}  |\tilde{\varphi}^{\prime}_{n} (|u_{\kappa}|^{2})| |\rho_{\kappa} \theta_{\kappa}|
|{\rm div} u_{\kappa}| dx dt\\
&\leq  2 \int_{0}^{T} \int_{\Omega} \psi(t) |\tilde{\varphi}^{\prime}_{n} (|u_{\kappa}|^{2})| |\rho_{\kappa}|
|\mathbb{D} u_{\kappa}|^{2} dx dt\\
&+ C(\|\psi\|_{L^{\infty}}) \int_{0}^{T} \int_{\Omega} |\tilde{\varphi}^{\prime}_{n} (|u_{\kappa}|^{2})| \rho_{\kappa} \theta_{\kappa}^{2}  dx dt.
\end{aligned}
\end{equation}
Thus,
\begin{equation}
\begin{aligned}
&|P| \leq 4 \|\psi\|_{L^{\infty}}  \int_{0}^{T} \int_{\Omega} \rho_{\kappa} |\nabla u_{\kappa}|^{2} dx dt\\
& + 2 \int_{0}^{T} \int_{\Omega} \psi(t)  |\tilde{\varphi}^{\prime}_{n} (|u_{\kappa}|^{2})| \rho_{\kappa}
|\mathbb{D} u_{\kappa}|^{2} dx dt\\
&+ C(\|\psi\|_{L^{\infty}}) \int_{0}^{T} \int_{\Omega}(1+ \tilde{\varphi}^{\prime}_{n} (|u_{\kappa}|^{2})) \rho_{\kappa} \theta_{\kappa}^{2}  dx dt.
\end{aligned}
\end{equation}
The first right hand side term will be controlled by
\begin{equation}
4 \| \psi \|_{L^{\infty}} (\int_{\Omega} (\rho_{0}|u_{0}|^{2} + \frac{\rho^{\gamma}_{0}}{\gamma-1})+ |\nabla \sqrt{\rho_{0}}|^{2}- r_{0}\log_{-} \rho_{0}) dx + 2 E_{0}
\end{equation}
and the second right hand side term will be absorbed by the dispersion term $A_{1}$ in \eqref{6.17}. By Lemma 6.1, we have
\begin{equation}
\int_{0}^{T} \int_{\Omega}(1+ \tilde{\varphi}^{\prime}_{n} (|u_{\kappa}|^{2})) \rho_{\kappa} \theta_{\kappa}^{2}  dx dt\rightarrow
\int_{0}^{T} \int_{\Omega}(1+ \tilde{\varphi}^{\prime}_{n} (|u|^{2})) \rho \theta^{2}  dx dt
\end{equation}
as $\kappa \rightarrow 0$.

Note that
\begin{equation}
\int_{0}^{T} \int_{\Omega} \psi(t) \varphi^{\prime}(u_{\kappa}) (r_{0} u_{\kappa} + r_{1} \rho_{\kappa} |u_{\kappa}|^{2} u_{\kappa}) dx dt \geq 0,
\end{equation}
so this term can be dropped directly.

Finally
\begin{equation}
\begin{aligned}
&\kappa \int_{0}^{T} \int_{\Omega} |\psi(t) \varphi_{n}(v_{\kappa}) \nabla \sqrt{\rho_{\kappa}}  \Delta \sqrt{\rho_{\kappa}} | dx dt\\
& \leq 2C(n,\psi)\kappa^{\frac{1}{4}} (\kappa^{\frac{1}{4}} \|\nabla \rho_{\kappa}^{\frac{1}{4}}\|_{L^{4}(0,T; L^{4}(\Omega))}) \\
& \| \sqrt{\kappa} \Delta \sqrt{\rho_{\kappa}}\|_{L^{2}(0,T; L^{2}(\Omega))}  \| \rho^{\frac{1}{4}}_{\kappa}\|_{L^{4}(0,T; L^{4}(\Omega))}\\
& \leq 2C(n,\psi) \kappa^{\frac{1}{4}} \rightarrow 0,
\end{aligned}
\end{equation}
as $\kappa \rightarrow 0$.

For the term $\mathbb{S}_{\kappa}= \phi_{K}(\rho_{\kappa}) \rho_{\kappa} (\mathbb{D} u_{\kappa}+ \kappa \frac{\Delta \sqrt{\rho_{\kappa}}}{\sqrt{\rho_{\kappa}}})= \mathbb{S}_{1} + \mathbb{S}_{2}$, we calculate as follows
\begin{equation}
\begin{aligned}
&\kappa \int_{0}^{T} \int_{\Omega} \psi(t) \mathbb{S}_{1} : \nabla (\varphi^{\prime}(u_{\kappa})) dx dt= \int_{0}^{T} \int_{\Omega} \psi(t)  \mathbb{D} u_{\kappa} : \nabla (\varphi^{\prime}_{n}(u_{\kappa})) dx dt\\
&=  \int_{0}^{T} \int_{\Omega} \psi(t) [\nabla u_{\kappa} \varphi^{\prime \prime}(v_{\kappa}) \rho_{\kappa}] : \mathbb{D} u_{\kappa}dx dt=  2\int_{0}^{T} \int_{\Omega} \psi(t)  \varphi^{\prime}(u_{\kappa}) \rho_{\kappa}\mathbb{D} u_{\kappa}: \nabla u_{\kappa} dx dt\\
&+ 4\int_{0}^{T} \int_{\Omega} \psi(t)  \varphi^{\prime\prime}(u_{\kappa}) \rho_{\kappa}
(\nabla u_{\kappa} u_{\kappa} \otimes u_{\kappa}) : \mathbb{D} u_{\kappa} dx dt\\
&= A_{1} +¡¡A_{2}.
\end{aligned}
\end{equation}
Notice that
\begin{equation*}
\mathbb{D} u_{\kappa} : \nabla u_{\kappa} = |\mathbb{D} u_{\kappa}|^{2},
\end{equation*}
thus
\begin{equation}
\begin{aligned}\label{6.17}
&A_{1}\geq 2\int_{0}^{T} \int_{\Omega} \psi(t)  \varphi^{\prime}(v_{\kappa}) (\phi_{K}(\rho_{\kappa}))^{2} \rho_{\kappa}|\mathbb{D} u_{\kappa}|^{2}dx dt\\
& - 4 \|\psi\|_{L^{\infty}} \int_{0}^{T} \int_{\Omega} \rho_{\kappa} |\nabla u_{\kappa}|^{2} dx dt,
\end{aligned}
\end{equation}
where we control $A_{2}$
\begin{equation}
\begin{aligned}
&A_{2}\leq 4\int_{0}^{T} \int_{\Omega}|\psi(t)| \frac{|v_{\kappa}|^{2}}{1+|v_{\kappa}|^{2}} \rho_{\kappa} |\nabla u_{\kappa}|^{2} dx dt\\
&\leq 4 \|\psi \|_{L^{\infty}}\int_{0}^{T} \int_{\Omega} \rho_{\kappa} |\nabla u_{\kappa}|^{2} dx dt \\
&\leq 4 \|\psi\|_{L^{\infty}}( (\int_{\Omega} (\rho_{0}|u_{0}|^{2} + \frac{\rho^{\gamma}_{0}}{\gamma-1})+ |\nabla \sqrt{\rho_{0}}|^{2}- r_{0}\log_{-} \rho_{0}) dx + 2 E_{0}),
\end{aligned}
\end{equation}
We need to treat the term related to $\mathbb{S}_{2}$,
\begin{equation}
\kappa \int_{0}^{T} \int_{\Omega} \psi \mathbb{S}_{2}: \nabla (\varphi^{\prime}_{n} (u_{\kappa})) dx dt = \kappa  \int_{0}^{T} \int_{\Omega} \psi \nabla u_{\kappa} \varphi^{\prime \prime}_{n}(u_{\kappa}): \sqrt{\rho_{\kappa}} \Delta  \sqrt{\rho_{\kappa}} dx dt=B,
\end{equation}
we control $B$ as follows
\begin{equation}
\begin{aligned}
&|B| \leq C(n,\psi) \|\sqrt{\rho_{\kappa}} \nabla u_{\kappa} \|_{L^{2}(0,T; L^{2}(\Omega))} \| \sqrt{\kappa} \Delta \sqrt{\rho_{\kappa}} \|_{L^{2}(0,T; L^{2}(\Omega))} \sqrt{\kappa} \\
& \leq C \kappa^{\frac{1}{2}}\rightarrow 0
\end{aligned}
\end{equation}
as $\kappa \rightarrow 0$.

With $(6.4)-(6.20)$, in particularly, letting $\kappa \rightarrow 0$ in \eqref{5.22}, dropping the positive terms on the left side, we hav the following inequality
\begin{equation}
\begin{aligned}
& - \int_{0}^{T} \int_{\Omega} \psi^{\prime}(t) \rho \varphi_{n}(u) dx dt \\
& \leq 4 \|\psi\|_{L^{\infty}}( (\int_{\Omega} (\rho_{0}|u_{0}|^{2} + \frac{\rho^{\gamma}_{0}}{\gamma-1})+ |\nabla \sqrt{\rho_{0}}|^{2}- r_{0}\log_{-} \rho_{0}) dx + 2 E_{0})\\
& + \psi(0) \int_{\Omega} \rho_{0} \varphi_{n} (u_{0}) dx + C(\|\psi\|_{L^{\infty}}) \int_{0}^{T} \int_{\Omega} (1+\tilde{\varphi}^{\prime}_{n}(|u|^{2}) ) \rho \theta^{2} dxdt,
\end{aligned}
\end{equation}
and
\begin{equation}
\begin{aligned}
& C(\|\psi\|_{L^{\infty}}) \int_{0}^{T} \int_{\Omega} (1+\tilde{\varphi}^{\prime}_{n}(|u|^{2}) ) \rho \theta^{2} dxdt \\
&\leq  C(\|\psi\|_{L^{\infty}}) \int_{0}^{T} (\int_{\Omega} (\rho \theta^{2})^{\frac{2}{2-\delta}}dx)^{\frac{2}{2-\delta}}
\times  (\int_{\Omega} (1+\tilde{\varphi}^{\prime}_{n}(|u|^{2}))^{\frac{2}{\delta}}dx) ^{\frac{\delta}{2}} dt  ,
\end{aligned}
\end{equation}
which in turn gives us Lemma 6.2.
\end{proof}

\section{Limit when $n \rightarrow \infty$, $r_{0} \rightarrow 0$ and $r_{1} \rightarrow 0$}
Thanks to the total energy estimate, thermal energy estimate, B-D entropy energy estimate, Mellet-Vasseur estimate, we have enough compactness to pass the final three parameter limit to get the weka solutions foe the Navier-Stokes-Fouier equations. The passage limit process is similar to the one in Section 5,6, so omit the details, here. Therefore we complete the proof of Theorem 1.2.

\phantomsection

\end{document}